\documentclass[11pt,a4paper]{article}
\usepackage[T1]{fontenc}
\usepackage[utf8]{inputenc}
\usepackage{amsmath,amssymb,amsthm,mathtools}

\usepackage{lmodern}

\usepackage[a4paper,margin=24mm]{geometry}
\usepackage{microtype}
\usepackage{footmisc}

\usepackage{booktabs,array,tabularx}
\usepackage{enumitem}
\usepackage{graphicx}
\usepackage{caption}
\usepackage{float}
\usepackage{needspace}
\usepackage{xcolor}

\numberwithin{equation}{section}
\usepackage[hypertexnames=false]{hyperref}
\usepackage[nameinlink,noabbrev]{cleveref}

\definecolor{CitationBlue}{HTML}{1F4E79}
\definecolor{ReferenceTeal}{HTML}{123B5D}

\hypersetup{
  colorlinks=true,
  citecolor=CitationBlue,
  linkcolor=ReferenceTeal,
  urlcolor=CitationBlue,
}

\newtheorem{theorem}{Theorem}[section]

\newtheorem{lemma}[theorem]{Lemma}
\newtheorem{corollary}[theorem]{Corollary}

\theoremstyle{remark}
\newtheorem{remark}[theorem]{Remark}

\theoremstyle{definition}
\newtheorem{example}[theorem]{Example}

\AddToHook{env/theorem/before}{\Needspace{3\baselineskip}}
\AddToHook{env/proposition/before}{\Needspace{3\baselineskip}}
\AddToHook{env/lemma/before}{\Needspace{3\baselineskip}}
\AddToHook{env/corollary/before}{\Needspace{3\baselineskip}}
\AddToHook{env/remark/before}{\Needspace{3\baselineskip}}
\AddToHook{env/example/before}{\Needspace{3\baselineskip}}

\floatstyle{ruled}
\newfloat{algorithm}{tbp}{loa}
\floatname{algorithm}{Algorithm}

\newcommand{\R}{\mathbb{R}}
\newcommand{\ip}[2]{\left\langle #1,#2\right\rangle}
\newcommand{\norm}[1]{\| #1\|}
\newcommand{\prox}{\operatorname{prox}}
\newcommand{\dom}{\operatorname{dom}}

\newcommand{\zer}{\operatorname{zer}}
\newcommand{\dist}{\operatorname{dist}}

\allowdisplaybreaks[2]
\setlist[enumerate]{
  leftmargin=*,
  itemsep=2pt,
  topsep=4pt,
  parsep=0pt
}

\AddToHook{cmd/normalsize/after}{%
  \setlength{\abovedisplayskip}{8pt plus 2pt minus 2pt}%
  \setlength{\belowdisplayskip}{8pt plus 2pt minus 2pt}%
  \setlength{\abovedisplayshortskip}{3pt plus 1pt}%
  \setlength{\belowdisplayshortskip}{5pt plus 1pt minus 1pt}%
}
\title{GRAAL for variational inequalities revisited:
  sharp \protect\\
  parameter range and adaptive step sizes
}

\author{%
  Santanu Soe\thanks{%
    Department of Mathematics, Indian Institute of Technology Madras,
    Chennai 600036, India; and School of Mathematics and Statistics,
    The University of Melbourne, Parkville, Victoria 3010, Australia.
    \newline
    Email:
    \href{mailto:ma22d002@smail.iitm.ac.in}%
      {\mbox{\nolinkurl{ma22d002@smail.iitm.ac.in}}};
    \href{mailto:santanu.soe@student.unimelb.edu.au}%
      {\mbox{\nolinkurl{santanu.soe@student.unimelb.edu.au}}}.%
  }
  \and
  Matthew K. Tam\thanks{%
    School of Mathematics and Statistics, The University of Melbourne,
    Parkville, Victoria 3010, Australia.
    \newline
    Email:
    \href{mailto:matthew.tam@unimelb.edu.au}%
      {\mbox{\nolinkurl{matthew.tam@unimelb.edu.au}}}.%
  }  
    \and
  V. Vetrivel\thanks{%
     Department of Mathematics, Indian Institute of Technology Madras,
    Chennai 600036, India.
    \newline
    Email:
    \href{mailto:vetri@iitm.ac.in}%
      {\mbox{\nolinkurl{vetri@iitm.ac.in}}}.%
  }  
}

\date{}

\begin{document}
\maketitle

\begin{abstract}
This paper revisits the Golden Ratio Algorithm (GRAAL) of Malitsky
(Math. Program. 184, 383--410, 2020) for monotone variational inequalities. The original analysis of GRAAL restricts the convex-combination parameter \(c\) to \(1<c\leq\varphi\), where \(\varphi=(1+\sqrt5)/2\) is the golden ratio. Through a new Lyapunov analysis, we show that this range of $c$ can be extended to \(1<c<3\). Further, a two-dimensional linear counterexample shows that this parameter range is sharp. In particular, for every \(c\geq3\) and every positive constant step size, the iterates of GRAAL can be unbounded. Building on this analysis, we propose a simpler version of the adaptive GRAAL (aGRAAL) that requires neither a line search nor a prescribed upper bound on the step sizes, and demonstrate its convergence for locally Lipschitz cost operators over the same parameter range. For both methods, we establish an \(O(1/N)\) ergodic gap bound and, under a residual-based error bound condition, prove the R-linear convergence of the adaptive iterates. Finally, numerical experiments are conducted to compare the proposed strategies with existing ones.
\end{abstract}


\section{Introduction}
\label{sec:introduction}
Let \(\mathbb H\) be a finite-dimensional real Hilbert space with
inner product \(\ip{\cdot}{\cdot}\) and induced norm
\(\norm{\cdot}\). In this paper, we consider the \emph{mixed variational inequality
problem} (MVIP):
\begin{equation}
\text{find }x^*\in\mathbb H\text{ such that }
\ip{F(x^*)}{x-x^*}+g(x)-g(x^*)\geq0
\quad\forall x\in\mathbb H.
\label{eq:monotone_vi}
\end{equation}
Throughout the paper, we assume that
\begin{enumerate}[label=\textnormal{(C.\arabic*)},ref=C.\arabic*]
\item\label{ass:operators}
\(g:\mathbb H\to(-\infty,+\infty]\) is a proper,
lower semicontinuous (lsc), and convex function. The operator
\(F:\overline{\dom g}\to\mathbb H\) is monotone, that is,
\[
\ip{F(x)-F(y)}{x-y}\geq0\qquad\forall x,y\in\overline{\dom g}.
\]
\item\label{ass:solution}
The solution set \(\mathcal S\) of \eqref{eq:monotone_vi} is nonempty.
\end{enumerate}
Additional assumptions on \(F\) will be stated when needed. The problem~\eqref{eq:monotone_vi} covers a wide range of models from optimisation, game theory,
and equilibrium analysis, as well as finite-dimensional discretisations
of variational inequalities in PDEs and
optimal control
\cite{FacchineiPang2003,kinderlehrer2000introduction,barbu1984optimal}. Here \(g\) typically represents constraints and/or nonsmooth terms,
while \(F\) need not be the gradient of a scalar function. We recall three
standard instances of \eqref{eq:monotone_vi} that motivate us to study this formulation.

\medskip
\textbf{Variational inequalities.}
Let \(C\subset\mathbb H\) be nonempty, closed, and convex, and take
\(g=\delta_C\), the indicator function of \(C\). Then
\eqref{eq:monotone_vi} reduces to
\[
\text{find }x^*\in C\text{ such that }
\ip{F(x^*)}{x-x^*}\geq0\qquad\forall x\in C.
\]
This is the classical variational inequality
problem~\cite{FacchineiPang2003}, which has been widely studied in the literature; see, for example,~\cite{patriksson1998nonlinear,solodov1999new,chan1982generalized}.

\medskip
\textbf{Composite minimisation.}
Let \(f:\mathbb H\to\R\) be convex and continuously differentiable. A point \(x^*\) solves
\[
\min_{x\in\mathbb H}\;f(x)+g(x)
\]
if and only if \(0\in\nabla f(x^*)+\partial g(x^*)\).
By the subgradient inequality, this inclusion is equivalent to
\eqref{eq:monotone_vi} with \(F=\nabla f\)
\cite{BauschkeCombettes2017}. For instance, the LASSO problem fits this formulation with  \(F(x)=A^*(Ax-b)\) and $g(x)=\gamma\|x\|_1$ ($\gamma>0$), where
\(A:\mathbb R^n\to\mathbb R^m\) is linear and \(b\in\mathbb R^m\).

\medskip
\textbf{Saddle point problems.}
Let \(g_i:\mathbb H_i\to(-\infty,+\infty]\), \(i=1,2\), be proper
lsc convex functions on finite-dimensional real Hilbert spaces. Let
\(\Phi:\mathbb H_1\times\mathbb H_2\to\R\) be continuously
differentiable, convex in its first argument, and concave in its second.
Consider
\begin{equation}
\min_{u\in\dom g_1}\;\max_{v\in\dom g_2}
g_1(u)+\Phi(u,v)-g_2(v).
\label{eq:saddle_problem}
\end{equation}
Here a solution is understood to be a saddle point. The saddle-point optimality conditions show that
\eqref{eq:saddle_problem} is equivalent to \eqref{eq:monotone_vi} with \(x=(u,v)\in\mathbb{H}=\mathbb{H}_1\times\mathbb{H}_2\),
\[
F(u,v)=\bigl(\nabla_u\Phi(u,v),-\nabla_v\Phi(u,v)\bigr)
\quad\text{and}\quad
g(u,v)=g_1(u)+g_2(v),
\]
see, for instance, \cite{nemirovski2004prox,BauschkeCombettes2017}.
The bilinear coupling \(\Phi(u,v)=\ip{Ku}{v}\), with \(K\) linear, is central to many primal--dual methods for convex optimisation
\cite{chambolle2011first,chambolle2016ergodic,chang2021golden}. Furthermore, methods for smooth nonbilinear couplings are also studied in
\cite{hamedani2021primal,chang2026convex}, and a decentralised
min--max setting is considered in \cite{malitsky2026first}.

When \(F\) is globally Lipschitz, popular methods for solving \eqref{eq:monotone_vi} include Tseng’s forward--backward--forward method~\cite{tseng2000modified} and forward--reflected--backward splitting~\cite{MalitskyTam2020}; see also \cite{cevher2021reflected,csetnek2019shadow,rieger2020backward} and the references therein. In most cases, the step sizes of these methods depend on a global Lipschitz constant to guarantee convergence of the iterates. However, computing such global constants can be expensive or unavailable (e.g., when $F$ is locally Lipschitz). Furthermore, since this constant measures the largest rate of change of \(F\) over its entire domain, the resulting step size can be overly conservative for iterates where \(F\) changes much more slowly. These observations motivate adaptive rules that select step sizes based on local information generated during each iteration \cite{malitsky2020golden,ou2026adaptive}. We focus on the golden ratio algorithm (GRAAL) and its adaptive
counterpart (aGRAAL) \cite[Algorithm 1]{malitsky2020golden}. Their common iteration reads as
\begin{equation}\label{eq:intro_graal}
\bar x^k=\frac{c-1}{c}x^k+\frac1c\bar x^{k-1},
\qquad
x^{k+1}=\prox_{\lambda_k g}
\bigl(\bar x^k-\lambda_kF(x^k)\bigr),
\end{equation}
where, for the fixed step size case (e.g., GRAAL), it is understood that
$\lambda_k\equiv\lambda$ for all $k$.
It is worth mentioning that the computational cost of \eqref{eq:intro_graal} is one new evaluation of \(F\) and one proximal evaluation per iteration. For \(c=\varphi:=(1+\sqrt5)/2\), Malitsky
\cite[Theorem~1]{malitsky2020golden} proved convergence of GRAAL when the step size satisfies \(0<\lambda\leq\varphi/(2L)\).
In this context, the golden ratio \(\varphi\) arises in the analysis due to the requirement that \(c-1-1/c\leq0\), which restricts \(c\) to \((1,\varphi]\). Subsequent work has been done to extend this range. For example, when \(g=\delta_C\) and \(c=2\), Alacaoglu et al.
\cite[Corollary~2]{alacaoglu2023beyond} proved convergence
of GRAAL and established an \(O(1/k)\) ergodic rate of convergence. Furthermore, Tran et al.~\cite[Theorem~6.2]{tran2026revisiting}
considered the resolvent version of \eqref{eq:intro_graal} for
the inclusion \(0\in F(x)+T(x)\), where \(T\) is maximally
\(3\)-cyclically monotone. It is known that the subdifferential of a proper lsc convex function is maximally monotone and cyclically monotone \cite[Theorem 22.18]{BauschkeCombettes2017}, and therefore maximally \(3\)-cyclically monotone. Thus, this includes the proximal case \(T=\partial g\). With suitable constant step sizes,
they established an \(O(k^{-1/2})\) best-iterate residual rate
for \(1<c<1+\sqrt3\). Interested readers are referred to \cite{chang2022grpda,chang2025splitting,chang2026convex, soe2026golden} for related parameter relaxations for structured primal--dual methods in both fixed and adaptive cases.

When $F$ is locally Lipschitz, aGRAAL chooses its step size by
\begin{equation}\label{eq:intro_agraal_step_size}
\lambda_k=\min\left\{
    \rho\lambda_{k-1},\;
    \frac{c\vartheta_{k-1}}{4\lambda_{k-1}}
    \frac{\norm{x^k-x^{k-1}}^2}
         {\norm{F(x^k)-F(x^{k-1})}^2},\;
    \lambda_{\max}
\right\},
\qquad
\vartheta_k=\frac{c\lambda_k}{\lambda_{k-1}},
\end{equation}
where \(1<c\leq\varphi\), \(\rho=1/c+1/c^2\),
\(\vartheta_0=1\), and \(\lambda_{\max}>0\). For $1<c<\varphi$, we have $\rho>1$, so the rule allows the step sizes to increase or decrease without a line search. The choice of $\lambda_{\max}$, however, requires some care. On the one hand, if it is too small, it may restrict the step sizes. On the other hand, a large value enlarges the constant in the available gap estimate. Specifically, aGRAAL has an \(O(\lambda_{\max}L^2/k)\) ergodic gap rate \cite[Section 4]{alacaoglu2023beyond}, where \(L\) is a Lipschitz constant of \(F\) on the bounded region containing the iterates.
Thus, increasing \(\lambda_{\max}\) can weaken the estimate even when the
upper bound does not affect the computed steps. Taking
\(\lambda_{\max}\) proportional to \(1/L\) gives an \(O(L/k)\) rate, but requires information about \(L\). Alacaoglu et al.~\cite[Algorithm~1]{alacaoglu2023beyond} addressed this issue by removing the upper bound \(\lambda_{\max}\) from \eqref{eq:intro_agraal_step_size} for the case when $g=\delta_C$. With a suitably chosen initialisation, they established an $\mathcal{O}(L/k)$ ergodic gap rate \cite[Theorem 5]{alacaoglu2023beyond} by showing that the sum of the step sizes grows at least linearly with the iteration count. They did not, however, prove convergence of the iterates since establishing a positive lower bound for $\vartheta_k$ in~\eqref{eq:intro_agraal_step_size} remains difficult.

The above discussion suggests two questions regarding \eqref{eq:intro_graal}--\eqref{eq:intro_agraal_step_size}. 
\begin{enumerate}[label=\textnormal{(\roman*)}]
\item How far can the range of $c$ be extended such that the iterates of GRAAL converge?
\item Can a simple adaptive rule give convergence of the sequences generated by \eqref{eq:intro_graal} without requiring the rule to include an absolute upper bound on the step sizes?
\end{enumerate}
By ``simple adaptive rule'', we mean a rule in which the variable step size approximates the inverse of the local Lipschitz constant of $F$, with no requirement that the step sizes be monotone. In this work, we address both of the aforementioned questions through an alternative convergence analysis of GRAAL, based on a Lyapunov function that covers both the fixed and the adaptive cases.

\subsection{Contributions}\label{sec:contributions}
To answer the first question, we revisit the standard proximal GRAAL
iteration~\eqref{eq:intro_graal} and prove convergence for every $c\in(1,3)$, assuming that $F$ is $L$-Lipschitz and $0<\lambda L<\mu(c)$, with $\mu(c)$ given by~\eqref{eq:fixed_final_step_size}. The corresponding algorithm with the extended parameter range is presented
in Algorithm~\ref{Alg:fixed}. We further show that the upper threshold is
sharp: for every $c\geq3$ and every positive constant step size, we
construct a counterexample in $\R^2$ whose iterates are unbounded. Moreover, we establish an $O(1/N)$ ergodic gap bound, and when \(g=\delta_C\), we express this bound in terms of the initial distance to a solution and then compare it with the estimate obtained in~\cite[Corollary~2]{alacaoglu2023beyond}.

\medskip
To answer the second question, we propose an adaptive version of
Algorithm~\ref{Alg:fixed} for locally Lipschitz \(F\) defined on
\(\overline{\dom g}\). The adaptive method is presented in
Algorithm~\ref{Alg:general}, which uses a simpler step size rule (see \eqref{eq:local_step_size})
than that of aGRAAL. More specifically, the second candidate
in~\eqref{eq:local_step_size} approximates the inverse of
a local Lipschitz constant of \(F\), whereas the corresponding
candidate in aGRAAL also depends on an earlier step size; see
\cite[Eq.~(21)]{malitsky2020golden}. As in aGRAAL, the first candidate is designed to allow the step size to grow, so the method can take larger steps where the curvature is flat. The adaptive method (Algorithm \ref{Alg:general}) admits the same range of $c$ as  Algorithm~\ref{Alg:fixed} and  requires neither a line search nor
a prescribed upper bound on the step sizes. We prove convergence of Algorithm~\ref{Alg:general} and establish an \(O(1/N)\) ergodic gap bound, together with R-linear convergence under an additional residual error bound condition. One concrete example of Algorithm~\ref{Alg:general} is
\begin{equation}
\boxed{\begin{aligned}
\lambda_k&=\min\left\{\frac54\lambda_{k-1},
\frac{11}{40}\frac{\norm{x^k-x^{k-1}}}
{\norm{F(x^k)-F(x^{k-1})}}\right\},\\[1mm]
\bar x^k&=\frac{x^k+\bar x^{k-1}}2,\\[1mm]
x^{k+1}&=\prox_{\lambda_k g}\bigl(\bar x^k-\lambda_kF(x^k)\bigr).
\end{aligned}}
\label{eq:intro_adaptive_examples}
\end{equation}
Here \(x^0,x^1\in\dom g\) and \(\lambda_0>0\), and we set
\(\bar x^0=x^1\).

\medskip
The structure of the paper is as follows.

\medskip
\noindent
After recalling the necessary preliminaries, Section~\ref{sec:fixed} proves
the convergence of Algorithm~\ref{Alg:fixed} and presents the counterexample. Section~\ref{sec:rate} establishes the ergodic gap bound and examines its constant in the projection case (i.e., $g=\delta_C$), while Section~\ref{sec:adaptive_algorithm_local_case} presents the adaptive method and proves its convergence, its ergodic gap bound, and its R-linear rate under a natural error bound. In Section~\ref{sec:structured_vi_applications} we derive an adaptive algorithm for structured variational inequalities by applying
Algorithm~\ref{Alg:general} to a suitable product-space formulation.
Finally, Section~\ref{sec:numerical_results} reports numerical experiments
comparing our proposed strategies with existing ones.

\subsection{Preliminaries}
\label{sec:preliminaries}

Let \(g:\mathbb H\to(-\infty,+\infty]\) be a proper lsc convex function. We write \(\dom g:=\{x\in\mathbb H:g(x)<+\infty\}\)
for the \emph{domain} of $g$. The \emph{convex subdifferential} at \(x\in\dom g\) is defined by
\[
\partial g(x):=\{u\in\mathbb H:
g(y)\geq g(x)+\ip{u}{y-x}\ \forall y\in\mathbb H\},
\]
and \(\partial g(x):=\varnothing\) when \(x\notin\dom g\).
The operator \(\partial g\) is \emph{maximally monotone}, and its \emph{graph} is closed in \(\mathbb H\times\mathbb H\). We refer the reader to \cite{BauschkeCombettes2017} for
these and related facts. By the definition of the subdifferential, \eqref{eq:monotone_vi} is equivalent to
\(0\in F(x^*)+\partial g(x^*)\). Thus,
\(\mathcal S=\zer(F+\partial g)\), where for a set-valued operator \(A\),
we write \(\zer A:=\{x\in\mathbb H:0\in A(x)\}\). Given \(t>0\), the \emph{proximal mapping} of \(g\) with parameter \(t\) is
defined by
\[
\prox_{tg}(v):=
\operatorname*{argmin}_{x\in\mathbb H}
\left\{g(x)+\frac{1}{2t}\norm{x-v}^{2}\right\}.
\]
Under Assumption~\eqref{ass:operators}, \(\prox_{tg}\) is single-valued
on \(\mathbb H\) and \emph{firmly nonexpansive}. Its optimality condition reads as
\begin{equation}
p=\prox_{tg}(v)
\quad\Longleftrightarrow\quad
\ip{p-v}{y-p}
\geq t\bigl(g(p)-g(y)\bigr)
\quad\forall\,y\in\dom g.
\label{eq:prox_inequality}
\end{equation}
Note that \eqref{eq:prox_inequality} also yields
\(\prox_{tg}(v)\in\dom\partial g\subseteq\dom g\) for all
\(v\in\mathbb H\) and \(t>0\). 

For a nonempty set
\(C\subset\mathbb H\), \(\overline C\) denotes its closure and
\(\dist(x,C):=\inf_{y\in C}\norm{x-y}\) the distance from \(x\in\mathbb H\)
to \(C\). If \(C\) is closed and convex, its \emph{metric projection} \(P_C\)
is single-valued and \(\dist(x,C)=\norm{x-P_Cx}\). Moreover, the
\emph{indicator function} \(\delta_C\), which takes the value \(0\) on
\(C\) and \(+\infty\) elsewhere, satisfies
\(P_C=\prox_{t\delta_C}\) for every \(t>0\).

Given \(x,z\in\dom g\), define
\begin{equation*}
\Psi_z(x)
=
\ip{F(z)}{x-z}+g(x)-g(z).
\end{equation*}
Given \(z\in\dom g\), the map \(x\mapsto\Psi_z(x)\) is convex, and
\(\Psi_z(x)\geq0\) for every \(x\in\dom g\) whenever \(z\in\mathcal S\). We conclude this section with the following useful lemmas.

\begin{lemma}[{\protect\cite[Lemma~1]{malitsky2020golden}}]
\label{lem:useful_1}
Let \((u^k)\) be a bounded sequence in \(\mathbb H\). If
\(\lim_{k\to\infty}\norm{u^k-u}\) exists for every cluster
point \(u\) of \((u^k)\), then \((u^k)\) converges.
\end{lemma}
\begin{lemma}
\label{lem:norm_identities}
For any \(p,q,r\in\mathbb H\) and \(\alpha\in\R\),
\begin{align}
\ip{p-q}{p-r}
&=
\frac12\norm{p-q}^{2}
+\frac12\norm{p-r}^{2}
-\frac12\norm{q-r}^{2},
\label{eq:three_point_identity}\\
\norm{\alpha p+(1-\alpha)q}^{2}
&=
\alpha\norm{p}^{2}
+(1-\alpha)\norm{q}^{2}
-\alpha(1-\alpha)\norm{p-q}^{2}.
\label{eq:convex_combination_identity}
\end{align}
\end{lemma}


\section{Golden ratio algorithm with sharp parameter range}
\label{sec:fixed}
\label{sec:master}
Before giving the proposed algorithm, it is convenient to define
\begin{equation}\mu(c):=
\begin{cases}
\dfrac{c(2c+2-c^2)}{2(c+1)},&1<c\leq1+\sqrt2,\\[2mm]
\dfrac{c(3-c)}{2(c-1)^2},&1+\sqrt2<c<3.
\end{cases}
\label{eq:fixed_final_step_size}
\end{equation}
\begin{algorithm}[H]
\caption{Extended GRAAL (E-GRAAL) for~\eqref{eq:monotone_vi}.}\label{Alg:fixed}
Choose \(x^0\in\dom g\), \(1<c<3\), and \(0<\lambda <\mu(c)/L\), where $L>0$ is the Lipschitz constant of $F$.
Set \(\bar x^{-1}=x^0\). For \(k\geq0\), compute
\begin{subequations}\label{eq:GRAAL_iterations}
\begin{align}
\bar x^k&=\frac{c-1}{c}x^k+\frac1c\bar x^{k-1},
\label{eq:convex_combination}\\
x^{k+1}&=\prox_{\lambda g}\bigl(\bar x^k-\lambda F(x^k)\bigr).
\label{eq:fixed_proximal_update}
\end{align}
\end{subequations}
\end{algorithm}

\noindent
For $\eta\ge0$, set \(\beta:=1+\frac1c-\eta\) and, for \(z\in\mathbb H\), define
\begin{equation*}
\mathcal E_k(z):=
\frac c{c-1}\norm{\bar x^k-z}^2
+\eta\norm{x^k-\bar x^{k-1}}^2
+\lambda L\norm{x^k-x^{k-1}}^2,\qquad k\geq1,
\end{equation*}
and, whenever \(\lambda<\frac{\beta}{2L}\), define
\begin{equation}
\Theta:=
\frac{\beta(\eta+2/c)-1-2\lambda L(\eta+\beta/c^2)}
{\beta-2\lambda L}.
\label{eq:descent_coefficient}
\end{equation}
\begin{remark}
The original analysis of GRAAL expands the inner product in
\cite[Eq.~(13)]{malitsky2020golden} by substituting the identity
$x^k-\bar x^{k-1}=c(x^k-\bar x^k)$, and it is precisely this substitution
that forces the constraint $c-1-1/c\leq 0$. In this work, we take a different approach. First, applying the identity
\begin{equation*}
x^{k+1}-x^k
=\bigl(x^{k+1}-\bar x^k\bigr)-\tfrac{1}{c}\bigl(x^k-\bar x^{k-1}\bigr),
\end{equation*}
from \eqref{eq:convex_combination}, we estimate the inner product
$\ip{x^k-\bar x^{k-1}}{x^{k+1}-x^k}$. Furthermore, the inner product
involving $F$, given in~\eqref{eq:ineq_4}, is handled in
exactly the same manner, together with the global Lipschitz estimate
of $F$. Second, we bound the resulting inner product \eqref{eq:weighted_young_ineq}
by Young's inequality such that a Fej\'er term $\eta\norm{x^k-\bar x^{k-1}}^2$, for some $\eta\ge0$, appears in the Lyapunov energy at iterations $k$ and $k+1$. Consequently, descent hinges on the joint choice of $\eta$ and $\lambda$ rather than on the condition $c-1-\frac1c\le0$, and as a by-product we get the restriction on $c$; see
Appendix~\ref{app:E-GRAAL_parameter_calculation} for further details.
\end{remark}

\begin{lemma}
\label{lem:fixed_energy}
Suppose that \(F:\dom g\to\mathbb{H}\) is \(L\)-Lipschitz for some \(L>0\), and conditions~(\ref{ass:operators})--(\ref{ass:solution}) hold. Let $(x^k)$ and $(\bar{x}^k)$ be generated by Algorithm~\ref{Alg:fixed}, and suppose that $\lambda < \frac{\beta}{2L}$. Then, for every \(z\in\dom g\)
and \(k\geq1\),
\begin{equation}
\mathcal E_{k+1}(z)+2\lambda\Psi_z(x^k)
+\Theta\norm{x^k-\bar x^{k-1}}^2
\leq\mathcal E_k(z).
\label{eq:fixed_descent_expanded}
\end{equation}
\end{lemma}

\begin{proof}
Following the initial steps of the GRAAL analysis
\cite[Theorem~1]{malitsky2020golden}, we apply
\eqref{eq:prox_inequality} at steps \(k\) and \(k-1\), with \(z\) and \(x^{k+1}\), respectively, to obtain
\begin{align}
\ip{x^{k+1}-\bar x^k}{z-x^{k+1}}
+\lambda\ip{F(x^k)}{z-x^{k+1}}
&\geq\lambda\bigl(g(x^{k+1})-g(z)\bigr),
\label{eq:prox_ineq_z}\\
\ip{x^k-\bar x^{k-1}}{x^{k+1}-x^k}
+\lambda\ip{F(x^{k-1})}{x^{k+1}-x^k}
&\geq\lambda\bigl(g(x^k)-g(x^{k+1})\bigr).
\label{eq:prox_ineq_new}
\end{align}
Adding \eqref{eq:prox_ineq_z} and \eqref{eq:prox_ineq_new} followed by
using monotonicity of \(F\) in the form
\(\ip{F(x^k)}{x^k-z}\geq\ip{F(z)}{x^k-z}\), we obtain
\begin{multline}\label{eq:ineq_1}
\lambda\Psi_z(x^k)
\leq\ip{x^{k+1}-\bar x^k}{z-x^{k+1}}
+\ip{x^k-\bar x^{k-1}}{x^{k+1}-x^k}\\
+\lambda\ip{F(x^k)-F(x^{k-1})}{x^k-x^{k+1}}.
\end{multline}
From \eqref{eq:convex_combination}, we have
\(x^{k+1}-z=\frac{c}{c-1}(\bar x^{k+1}-z)-\frac1{c-1}(\bar x^k-z)\). Thus using \eqref{eq:convex_combination_identity} yields
\begin{equation}
\norm{x^{k+1}-z}^2
=\frac c{c-1}\norm{\bar x^{k+1}-z}^2
-\frac1{c-1}\norm{\bar x^k-z}^2
+\frac1c\norm{x^{k+1}-\bar x^k}^2.
\label{eq:identity_1}
\end{equation}
Furthermore, \eqref{eq:convex_combination} gives \(x^{k+1}-x^k=(x^{k+1}-\bar x^k)-\frac1c(x^k-\bar x^{k-1})\). Using this fact together with~\eqref{eq:three_point_identity},~\eqref{eq:convex_combination_identity}~and~\eqref{eq:identity_1}, we have
\begin{align}
2\ip{x^{k+1}-\bar x^k}{z-x^{k+1}}
&=\norm{\bar x^k-z}^2-\norm{x^{k+1}-z}^2
-\norm{x^{k+1}-\bar x^k}^2 \notag\\
&=\frac c{c-1}\left(
\norm{\bar x^k-z}^2-\norm{\bar x^{k+1}-z}^2\right)
-\left(1+\frac1c\right)\norm{x^{k+1}-\bar x^k}^2,
\label{eq:ineq_2}\\
2\ip{x^k-\bar x^{k-1}}{x^{k+1}-x^k}
&=2\ip{x^k-\bar x^{k-1}}{x^{k+1}-\bar x^k}
-\frac2c\norm{x^k-\bar x^{k-1}}^2.
\label{eq:ineq_3}
\end{align}
Substituting \eqref{eq:ineq_2} and
\eqref{eq:ineq_3} into \eqref{eq:ineq_1} gives
\begin{multline}
2\lambda\Psi_z(x^k)+\frac c{c-1}\norm{\bar x^{k+1}-z}^2
\leq\frac c{c-1}\norm{\bar x^k-z}^2
-\left(1+\frac1c\right)\norm{x^{k+1}-\bar x^k}^2\\+2\ip{x^{k+1}-\bar x^k}{x^k-\bar x^{k-1}}
-\frac2c\norm{x^k-\bar x^{k-1}}^2
+2\lambda\ip{F(x^k)-F(x^{k-1})}{x^k-x^{k+1}}.
\label{eq:ineq_4}
\end{multline}
By Lipschitz continuity of $F$ and Young's inequality, the last term in
\eqref{eq:ineq_4} can be bounded as
\begin{equation*}
2\lambda\ip{F(x^k)-F(x^{k-1})}{x^k-x^{k+1}}
\leq\lambda L\left(\norm{x^k-x^{k-1}}^2+\norm{x^{k+1}-x^k}^2\right).
\end{equation*}
Inserting this estimate into \eqref{eq:ineq_4} and adding
\(\lambda L\norm{x^{k+1}-x^k}^2\) to both sides followed by using the identity
\(x^{k+1}-x^k=(x^{k+1}-\bar x^k)-\frac1c(x^k-\bar x^{k-1})\) yields
\begin{multline}
\frac c{c-1}\norm{\bar x^{k+1}-z}^2
+\lambda L\norm{x^{k+1}-x^k}^2+2\lambda\Psi_z(x^k)
\leq\frac c{c-1}\norm{\bar x^k-z}^2
+\lambda L\norm{x^k-x^{k-1}}^2\\
-\left(1+\frac1c-2\lambda L\right)
\norm{x^{k+1}-\bar x^k}^2
+2\left(1-\frac{2\lambda L}{c}\right)
\ip{x^{k+1}-\bar x^k}{x^k-\bar x^{k-1}}\\
-\left(\frac2c-\frac{2\lambda L}{c^2}\right)
\norm{x^k-\bar x^{k-1}}^2.
\label{eq:ineq_5}
\end{multline}
Denote \(\tau:=\beta-2\lambda L>0\). To estimate the inner product in \eqref{eq:ineq_5}, we use Young's inequality to obtain
\begin{align}
2\left(1-\frac{2\lambda L}{c}\right)
\ip{x^{k+1}-\bar x^k}{x^k-\bar x^{k-1}}
&\leq\tau\norm{x^{k+1}-\bar x^k}^2+\frac{(1-2\lambda L/c)^2}{\tau}
\norm{x^k-\bar x^{k-1}}^2.
\label{eq:weighted_young_ineq}
\end{align}
By combining \eqref{eq:weighted_young_ineq} with \eqref{eq:ineq_5}, we obtain
\begin{multline}
\frac c{c-1}\norm{\bar x^{k+1}-z}^2
+\lambda L\norm{x^{k+1}-x^k}^2+2\lambda\Psi_z(x^k)
\leq\frac c{c-1}\norm{\bar x^k-z}^2
+\lambda L\norm{x^k-x^{k-1}}^2\\
+\left(\tau-1-\frac1c+2\lambda L\right)
\norm{x^{k+1}-\bar x^k}^2
+\left(\frac{(1-2\lambda L/c)^2}{\tau}-\frac2c+\frac{2\lambda L}{c^2}\right)
\norm{x^k-\bar x^{k-1}}^2.
\label{eq:ineq_5_1}
\end{multline}
Since \(\tau=1+\frac1c-\eta-2\lambda L\), we have $\tau-1-\frac1c+2\lambda L=-\eta.$ Recalling $\tau:=\beta-2\lambda L$, the coefficient of \(\norm{x^k-\bar x^{k-1}}^2\) is
\begin{align*}
\frac{(1-2\lambda L/c)^2}{\tau}-\frac2c+\frac{2\lambda L}{c^2}
&=\eta -\left(\eta+\frac2c-\frac{2\lambda L}{c^2}-\frac{(1-2\lambda L/c)^2}{\beta-2\lambda L}\right)\\
&=\eta-
\frac{\left(\eta+\frac2c-\frac{2\lambda L}{c^2}\right)(\beta-2\lambda L)
-\left(1-\frac{2\lambda L}{c}\right)^2}{\beta-2\lambda L}\\
&=\eta-
\frac{\beta(\eta+2/c)-1-2\lambda L(\eta+\beta/c^2)}
{\beta-2\lambda L}\\&
=\eta-\Theta.
\end{align*}
Consequently, combining the above facts with \eqref{eq:ineq_5_1} and then recalling the definition of \(\{\mathcal E_k(z)\}\) proves~\eqref{eq:fixed_descent_expanded}.
\end{proof}

\medskip
To guarantee convergence from Lemma~\ref{lem:fixed_energy}, we must choose
\(\eta\ge0\) and \(\lambda>0\) such that both the weight $\tau$ from Young's inequality and $\Theta$ are positive. For \(1<c<3\), set
\begin{equation}
\eta=
\begin{cases}
0,&1<c\leq1+\sqrt2,\\[2mm]
\dfrac{c^2-2c-1}{c(c-1)},&1+\sqrt2<c<3.
\end{cases}
\label{eq:fixed_eta_choice}
\end{equation}
As shown in Appendix~\ref{app:E-GRAAL_parameter_calculation}, this choice maximises the step size bound supplied by~\eqref{eq:descent_coefficient} over all \(\eta\geq0\). The resulting upper bound is \(\mu(c)\) in \eqref{eq:fixed_final_step_size}. Whenever $0 < \lambda L < \mu(c)$, the hypothesis of Lemma~\ref{lem:fixed_energy}, namely, $\lambda L < \beta/2$ is satisfied, and one also has $\Theta > 0$. The value of $\eta$ from \eqref{eq:fixed_eta_choice} is kept fixed throughout the convergence and ergodic rate analysis. With these preparations, we are now in a position to state the main convergence result of Algorithm~\ref{Alg:fixed}.

\begin{theorem}
\label{thm:fixed}
Suppose that \(F:\dom g\to\mathbb{H}\) is \(L\)-Lipschitz for some \(L>0\), and conditions~(\ref{ass:operators})--(\ref{ass:solution}) hold. Let \(1<c<3\) and take \(0<\lambda L<\mu(c)\), where \(\mu(c)\) is defined in~\eqref{eq:fixed_final_step_size}.
Then both sequences \((x^k)\) and \((\bar x^k)\) generated by Algorithm~\ref{Alg:fixed} converge
to the same solution of~\eqref{eq:monotone_vi}.
\end{theorem}

\begin{proof}
As shown in Appendix~\ref{app:E-GRAAL_parameter_calculation}, the choice of $\eta$ from \eqref{eq:fixed_eta_choice} gives
\(\beta>2\lambda L\) and \(\Theta>0\).
Now fix \(z^*\in\mathcal S\). Then \(\Psi_{z^*}(x^k)\geq0\), and hence by Lemma~\ref{lem:fixed_energy},
\begin{equation}
\mathcal E_{k+1}(z^*)+\Theta\norm{x^k-\bar x^{k-1}}^2
\leq\mathcal E_k(z^*).
\label{eq:convergence_descent_ineq}
\end{equation}
Therefore, it follows that \(\lim_{k\to\infty}\mathcal E_k(z^*)\) exists, and \((\bar x^k)\) is bounded.
Furthermore, using \(x^k=(c\bar x^k-\bar x^{k-1})/(c-1)\), we also obtain
boundedness of \((x^k)\). Moreover, by \eqref{eq:convergence_descent_ineq}, \(x^k-\bar x^{k-1}\to0\). Since  \eqref{eq:convex_combination} gives
\(x^k-\bar x^k=(x^k-\bar x^{k-1})/c\), we have
\(x^k-\bar x^k\to0\). Together, these facts yield
\(x^{k+1}-x^k=(x^{k+1}-\bar x^k)-(x^k-\bar x^k)\to0\). To identify the cluster points, take any subsequence $(x^{k_j})$ such that
\(x^{k_j}\to\tilde x\). By the proximal optimality condition, we have
\[
\frac{\bar x^{k_j}-x^{k_j+1}}{\lambda}-F(x^{k_j})
\in\partial g(x^{k_j+1}).
\]
By Lipschitz continuity of \(F\), the sequence \((F(x^{k_j}))\)
is Cauchy; denote its limit by \(v\). Since $x^{k_j+1}\to\tilde x$ and
$\bar x^{k_j}-x^{k_j+1}\to0$, closedness of the graph of $\partial g$
gives $-v\in\partial g(\tilde x)$. In particular,
\(\tilde x\in\dom g\). We may now use Lipschitz continuity at
\(\tilde x\) to conclude that \(v=F(\tilde x)\), and hence
\(\tilde x\in\mathcal S\). For every \(z^*\in\mathcal S\), the sequence
\(\{\mathcal E_k(z^*)\}\) has a limit, and its second and third terms tend to zero. Therefore
\(\lim_{k\to\infty}\norm{\bar x^k-z^*}\) exists.
Since the two sequences have the same cluster points, all in
\(\mathcal S\), Lemma~\ref{lem:useful_1} gives
convergence of~\((\bar x^k)\). Since \(\lim_{k\to\infty}(x^k-\bar x^k)=0\), the sequence \((x^k)\) converges to the same solution. This completes the proof.
\end{proof}

\begin{remark}
When \(c=\varphi\), our analysis recovers the step size range, \(0<\lambda L\leq\varphi/2\) of GRAAL, and its Lyapunov inequality~\cite[Theorem~1]{malitsky2020golden}.
At $c=\varphi$, $\eta=0$ and \(\mu(\varphi)=\varphi/2\). Therefore Theorem~\ref{thm:fixed}
covers the case \(0<\lambda L<\varphi/2\).
Furthermore, when $\lambda L = \varphi/2$, the conclusion follows directly from
\eqref{eq:ineq_5} in the proof of
Lemma~\ref{lem:fixed_energy}. Indeed, setting \(c=\varphi\) and
\(\lambda L=\varphi/2\) yields $1+\frac1c-2\lambda L= 1-\frac{2\lambda L}{c}=0$. Then using
\(x^k-\bar x^{k-1}=\varphi(x^k-\bar x^k)\), we obtain
\begin{multline*}
2\lambda\Psi_z(x^k)
+(\varphi+1)\norm{\bar x^{k+1}-z}^2
+\frac{\varphi}{2}\norm{x^{k+1}-x^k}^2\\
\leq(\varphi+1)\norm{\bar x^k-z}^2
+\frac{\varphi}{2}\norm{x^k-x^{k-1}}^2
-\varphi\norm{x^k-\bar x^k}^2,
\end{multline*}
which is the main Lyapunov inequality of GRAAL derived in \cite[Eq.~(20)]{malitsky2020golden}.
\end{remark}

\subsection{A counterexample for convergence when \texorpdfstring{\(c\geq3\)}{c >= 3}}

The next example shows that the upper bound of $c$ cannot be extended further for Algorithm~\ref{Alg:fixed}.
\begin{example}
\label{ex:divergence_example_with_fixed_step-size}
Let \(\mathbb H=\R^2\), \(g\equiv0\), and
\[
F:=\begin{bmatrix}0&-1\\1&0\end{bmatrix}.
\]
Then, for every \(c\geq3\), \(\lambda>0\), and \(x^0\neq0\), both sequences
generated by \eqref{eq:GRAAL_iterations} are unbounded.
\end{example}
\begin{proof}
The operator \(F\) is monotone and \(1\)-Lipschitz, and the unique
solution of \eqref{eq:monotone_vi} is zero. With \(I\) denoting the identity matrix on \(\R^2\),
the iteration \eqref{eq:GRAAL_iterations} takes the form \(z^{k+1}=Tz^k\), where
\[
z^k=\begin{pmatrix}x^k\\ \bar x^{k-1}\end{pmatrix}
\quad\text{and}\quad
T=\begin{pmatrix}
\dfrac{c-1}{c}I-\lambda F & \dfrac1c I\\[2mm]
\dfrac{c-1}{c}I & \dfrac1c I
\end{pmatrix}\in\R^{4\times4}.
\]
The characteristic polynomial of \(T\) is
\[
\chi_{T}(\zeta)
=P(\zeta)\overline{P(\bar\zeta)},
\]
where \(P(\zeta)=\zeta^2+a\zeta+b\) has
\(a=-1+\mathrm{i}\lambda\) and \(b=-\mathrm{i}\lambda/c\).
If both roots of \(P\) lie in the closed unit disc, then the Schur--Cohn inequality yields
\(\lvert a-\bar a b\rvert\leq1-\lvert b\rvert^2\);
see \cite[Theorem~6.8b]{henrici1974acca1}.
However, a direct calculation shows that for all \(c\geq3\) and \(\lambda>0\),
\begin{align*}
|a-\bar a b|^2-(1-|b|^2)^2
&=\left(1-\frac{\lambda^2}{c}\right)^2
+\lambda^2\left(1-\frac1c\right)^2
-\left(1-\frac{\lambda^2}{c^2}\right)^2\\
&=\frac{\lambda^2(c-1)(c-3)}{c^2}
+\frac{\lambda^4(c^2-1)}{c^4}>0.
\end{align*}
Therefore \(\rho(T)>1\). Let \(r\) be a root of \(P\) with \(|r|>1\).  Then it is easy to verify that
\(\ell=\bigl(cr-1,\mathrm{i}(cr-1),1,\mathrm{i}\bigr)\) is an
eigenvector of \(T\) with corresponding eigenvalue \(r\), that is \(\ell T=r\ell\). Since $x^0=(x^0_1, x^0_2)\neq0$, we have \(\ell z^0=cr(x_1^0+\mathrm{i}x_2^0)\neq0\). Therefore
\[
\|z^k\|\geq\frac{|\ell T^kz^0|}{\|\ell\|}
=\frac{|r|^k|\ell z^0|}{\|\ell\|}
\longrightarrow+\infty.
\]
Together with \eqref{eq:convex_combination}, this shows that both $(x^k)$ and $(\bar x^k)$ are unbounded.
\end{proof}
\begin{remark}
    Note that for \(g\equiv0\), eliminating \(\bar x^{k-1}\) gives, for \(k\geq1\),
\[
x^{k+1}=x^k-\lambda F(x^k)+\frac{\lambda}{c}F(x^{k-1})
=x^k-\gamma\bigl(F(x^k)+\alpha(F(x^k)-F(x^{k-1}))\bigr),
\]
where \(\gamma=\lambda(c-1)/c\) and \(\alpha=1/(c-1)\).
Thus \(c\ge3\) corresponds to \(0<\alpha\le1/2\), consistent with the counterexample for convergence of the $\rm FRB_{\alpha}$ algorithm as presented in~\cite[Example 2.1]{ou2026adaptive}.
\end{remark}

\section{Ergodic rate of convergence}
\label{sec:rate}
In this section, we derive an ergodic gap bound for Algorithm~\ref{Alg:fixed} and then compare its coefficient with~\cite[Corollary~2]{alacaoglu2023beyond}.
Given a nonempty compact set \(U\subset\dom g\), define the restricted
dual gap \cite{nemirovski2004prox,nesterov2007dual,malitsky2020golden} by
\begin{equation*}
\mathcal G_U(x):=\sup_{z\in U}\Psi_z(x)
=\sup_{z\in U}\left\{\ip{F(z)}{x-z}+g(x)-g(z)\right\}.
\end{equation*}
It is worth mentioning that \(\mathcal G_U\) is finite and
convex on \(\dom g\). Given \(x\in U\),
\(\mathcal G_U(x)\geq0\), and equality holds when \(x\in\mathcal U\cap S\). Moreover, if $U\cap S\neq\varnothing$, then $\mathcal G_U(x)\ge0$ for all $x\in\dom g$. Now, for \(N\geq1\), define the ergodic sequence
\begin{equation*}
X_N:=\frac1N\sum_{k=1}^N x^k.
\end{equation*}

\begin{theorem}
\label{thm:ergodic_rate}
Under the assumptions of Theorem~\ref{thm:fixed}, let \((x^k)\) and $(\bar x^k)$ be
generated by Algorithm~\ref{Alg:fixed}. Then, for any nonempty compact set \(U\subset\dom g\) and \(N\geq1\),
\begin{equation*}
\mathcal G_U(X_N)\leq\frac1{2\lambda N}\sup_{z\in U}
\left\{\frac c{c-1}\norm{\bar x^1-z}^2
+(\eta+\lambda L)\norm{x^1-x^0}^2\right\}.
\end{equation*}
\end{theorem}
\begin{proof}
Fix \(z\in U\). Summing \eqref{eq:fixed_descent_expanded} from
\(k=1\) to \(N\) gives
\[
2\lambda\sum_{k=1}^N\Psi_z(x^k)
+\mathcal E_{N+1}(z)
+\Theta\sum_{k=1}^N\norm{x^k-\bar x^{k-1}}^2
\leq\mathcal E_1(z).
\]
Since \(\mathcal E_{N+1}(z)\geq0\), \(\Theta>0\), and
\(\Psi_z\) is convex, we have
\begin{equation}
2\lambda N\Psi_z(X_N)
\leq2\lambda\sum_{k=1}^N\Psi_z(x^k)
\leq\mathcal E_1(z).
\label{eq:ergodic_rate_convex_ineq}
\end{equation}
Finally, using \(\bar x^0=x^0\) and then
taking the supremum over \(z\in U\) completes the proof.
\end{proof}

To express the bound in terms of the initial distance to a solution, we define
\begin{equation*}
C_1:=\frac c{c-1}+
\frac{(\lambda L)^2}{\beta-\lambda L}\quad\text{and}\quad
C_2:=\frac{(c+1)^2}{c(c-1)}C_1.
\end{equation*}
Note that the denominator in \(C_1\) is positive because
\(\beta>2\lambda L\).

\begin{lemma}
\label{lem:initial_energy}
Under the assumptions of Theorem~\ref{thm:fixed}, let \(z^*\in\mathcal S\). Then
\begin{equation}
\mathcal E_1(z^*)
\leq C_1\norm{x^0-z^*}^2.
\label{eq:initial_energy_bounded}
\end{equation}
\end{lemma}
\begin{proof}
From the first proximal step ($k=0$) and the solution criteria, we have
\[
x^1=\prox_{\lambda g}(x^0-\lambda F(x^0))
\quad\text{and}\quad
z^*=\prox_{\lambda g}(z^*-\lambda F(z^*)),
\]
respectively. Let \(v=x^0-z^*\), \(s=x^1-x^0\), and
\(d=F(x^0)-F(z^*)\). By the firm nonexpansiveness of the proximal
mapping, we obtain
\[
\norm{v+s}^2\leq\ip{v+s}{v-\lambda d},
\quad\text{or, equivalently,}\quad
\ip{v+s}{s+\lambda d}\leq0.
\]
By monotonicity and Lipschitz continuity of \(F\), we have
\(\ip{v}{d}\geq0\) and \(\norm d\leq L\norm v\). Hence
\begin{equation*}
\ip{v}{s}
\leq-\norm{s}^2-\lambda\ip{s}{d}
\leq-\norm{s}^2+\lambda L\norm v\norm s.
\end{equation*}
Since \(\bar x^1-z^*=v+\frac{c-1}{c}s\), it follows that
\begin{align*}
\mathcal E_1(z^*)
&=\frac c{c-1}\norm v^2+2\ip{v}{s}
+\left(\frac{c-1}{c}+\eta+\lambda L\right)\norm s^2
\notag\\
&\leq\frac c{c-1}\norm v^2
+2\lambda L\norm v\norm s-(\beta-\lambda L)\norm s^2
\notag\\
&=C_1\norm v^2
-(\beta-\lambda L)
\left(\norm s-\frac{\lambda L\norm v}{\beta-\lambda L}\right)^2.
\end{align*}
This proves \eqref{eq:initial_energy_bounded}.
\end{proof}

We now specialise to \(g=\delta_C\), where \(C\subset\mathbb H\) is a nonempty, closed, and convex set. Fix a solution \(z^*\in C\)
and consider the compact convex set
\begin{equation}
\mathcal B:=\left\{x\in C:
\norm{x-z^*}^2\leq18\norm{x^0-z^*}^2\right\}.
\label{eq:comparison_within_the_ball_B}
\end{equation}
We use this set to make a fair comparison of the gap rate with
\cite[Corollary~2]{alacaoglu2023beyond}.
Before estimating the gap on \(\mathcal B\), we verify that it
contains the iterates under a suitable bound on~\(C_2\).

\begin{lemma}
\label{lem:both_bdd_sequnce}
Under the assumptions of Lemma~\ref{lem:initial_energy}, suppose that \(g=\delta_C\), where \(C\) is nonempty, closed, and convex. Then, for every \(k\geq0\),
\begin{equation*}
\max\left\{\norm{x^k-z^*}^2,\norm{\bar x^k-z^*}^2\right\}
\leq C_2\norm{x^0-z^*}^2.
\end{equation*}
Consequently, if \(C_2\leq18\), then
\(x^k,\bar x^k,X_N\in\mathcal B\) for all \(k\geq0\) and
\(N\geq1\).
\end{lemma}
\begin{proof}
Since \(z^*\) is a solution, \(\Psi_{z^*}(x^k)\geq0\).
From \eqref{eq:fixed_descent_expanded} and Lemma~\ref{lem:initial_energy},
we have
\[
\norm{\bar x^k-z^*}^2
\leq\frac{c-1}{c}C_1\norm{x^0-z^*}^2.
\]
This estimate also holds at \(k=0\), since \(\bar x^0=x^0\)
and \(C_1\geq c/(c-1)\). For \(k\geq1\), it follows from~\eqref{eq:convex_combination} that
\begin{align*}
\norm{x^k-z^*}
&\leq\frac{c\norm{\bar x^k-z^*}
+\norm{\bar x^{k-1}-z^*}}{c-1}\leq\frac{c+1}{\sqrt{c(c-1)}}\sqrt{C_1}\,
\norm{x^0-z^*}.
\end{align*}
The same bound holds at \(k=0\), because \(C_2>1\), and it also
controls \(\bar x^k\), because
\(C_2\geq(c-1)C_1/c\). Finally, when \(C_2\leq18\), we have \(X_N\in\mathcal B\)
by convexity of \(\mathcal B\).
\end{proof}

\begin{theorem}
\label{thm:main_ergodic_rate_fixed_steps}
Under the assumptions of Lemma~\ref{lem:both_bdd_sequnce}, suppose
that \(C_2\leq18\), and let \(\mathcal B\) be defined by
\eqref{eq:comparison_within_the_ball_B}. Then, for every \(N\geq1\),
\begin{equation}
\mathcal G_\mathcal B(X_N)\leq
\frac{\left(\sqrt{\dfrac{18c}{c-1}}+
\sqrt{C_1}\right)^2}{2\lambda L}
\frac{L\norm{x^0-z^*}^2}{N}.
\label{eq:distance_rate_fixed_step_size}
\end{equation}
\end{theorem}
\begin{proof}
Let \(z\in\mathcal B\) and \(\tau>0\). To use the initial energy
bound at \(z^*\), we write
\(\bar x^1-z=(\bar x^1-z^*)+(z^*-z)\).
By Young's inequality and Lemma~\ref{lem:initial_energy}, we have
\begin{align}
\mathcal E_1(z)
&\leq\frac c{c-1}(1+\tau)\norm{z-z^*}^2
+\left(1+\frac1\tau\right)\mathcal E_1(z^*)
\notag\\
&\leq\left\{\frac{18c}{c-1}(1+\tau)
+C_1\left(1+\frac1\tau\right)\right\}
\norm{x^0-z^*}^2.
\label{eq:Young_energy_bound}
\end{align}
To minimise the right-hand side of \eqref{eq:Young_energy_bound}, note that, for
\(a,b>0\), the minimum of $a(1+\tau)+b(1+1/\tau)$ is attained at \(\tau=\sqrt{b/a}\). Therefore, by taking
\(a=18c/(c-1)\) and \(b=C_1\) in~\eqref{eq:Young_energy_bound}, we obtain
\begin{equation*}
\sup_{z\in\mathcal B}\mathcal E_1(z)
\leq\left(\sqrt{\frac{18c}{c-1}}+\sqrt{C_1}\right)^2
\norm{x^0-z^*}^2.
\end{equation*}
The conclusion follows from \eqref{eq:ergodic_rate_convex_ineq}.
\end{proof}

\begin{corollary}
\label{cor:ergodic_gap_constant_comparison}
In the setting of Theorem~\ref{thm:main_ergodic_rate_fixed_steps}, choose
\begin{equation}
c=\frac{19}{10},\qquad \eta=0,\qquad \lambda=\frac5{7L}.
\label{eq:selected_parameters}
\end{equation}
Then \(x^k,\bar x^k,X_N\in\mathcal B\) for all \(k\geq0\) and
\(N\geq1\), and
\begin{equation*}
\mathcal G_\mathcal B(X_N)\leq
\frac{42.802L\norm{x^0-z^*}^2}{N}.
\end{equation*}
\end{corollary}
\begin{proof}
The parameters in \eqref{eq:selected_parameters} satisfy
\(\lambda L=5/7<\mu(19/10)=4161/5800\). Furthermore, a direct substitution gives 
\[
\beta-2\lambda L=\frac{13}{133}>0,\quad
\Theta=\frac{127}{4693}>0,\quad
C_1=\frac{2071}{756},\quad\text{and}\quad
C_2=\frac{91669}{6804}<18.
\]
Therefore, the coefficient in \eqref{eq:distance_rate_fixed_step_size} is
\[
\frac{1}{2\lambda L}\left(\sqrt{\dfrac{18c}{c-1}}+
\sqrt{C_1}\right)^2
=\frac7{10}
\left(\sqrt{38}+\sqrt{\frac{2071}{756}}\right)^2
\approx42.802.
\]
\end{proof}

\begin{remark}
Let \(c=2\) and \(\alpha=(1-\varepsilon)/L\),
where \(0<\varepsilon<1\). Then \cite[Corollary 2 and Eq. (20)]{alacaoglu2023beyond} yields the following bound on \(\mathcal B\)
\begin{equation*}
\mathcal G_\mathcal B(X_N)\leq
\frac{C_0}{1-\varepsilon}\,
\frac{L\norm{x^0-z^*}^2}{N},
\end{equation*}
where
\begin{equation*}
\begin{aligned}
C_0
&=\frac12\min_{\tau>0}
\left\{36(1+\tau)+3\left(1+\frac1\tau\right)+30\right\}=\frac{69+12\sqrt3}{2}\approx44.892.
\end{aligned}
\end{equation*}
Thus, recomputing the constant from these estimates gives
\(C_0\approx44.892\), rather than the value \(32\)
reported in \cite[Corollary~2]{alacaoglu2023beyond}. However, the coefficient in Corollary~\ref{cor:ergodic_gap_constant_comparison}
provides a modest improvement over this recomputed constant.
\end{remark}

\section{Adaptive extended golden ratio algorithm}
\label{sec:adaptive_algorithm_local_case}
In this section, we present a fully adaptive variant of Algorithm~\ref{Alg:fixed}
whose convergence does not require an explicit upper bound on the step sizes. It involves three
parameters, \(c\), \(\rho\), and \(\sigma\), whose choice is discussed after
the algorithm is presented. Throughout, the step size rule \eqref{eq:local_step_size} is understood with the convention $\frac{a}{0} = +\infty$ for $a \geq 0$.
\begin{algorithm}[H]
\caption{Adaptive Extended GRAAL (AE-GRAAL) for \eqref{eq:monotone_vi}.}\label{Alg:general}
Let \(1<c<3\), and select \(\rho\), \(\sigma>0\) from \eqref{eq:select_rho_from_the_range}--\eqref{eq:select_sigma_from_the_bound}, respectively. Let \(x^0,x^1\in\dom g\), \(\lambda_0>0\), and set
\(\bar x^0=x^1\). For \(k\geq1\), compute

\begin{align}
\lambda_k&=\min\left\{\rho\lambda_{k-1},
\frac{\sigma\norm{x^k-x^{k-1}}}{\norm{F(x^k)-F(x^{k-1})}}\right\},
\label{eq:local_step_size}\\
\bar x^k&=\frac{c-1}{c}x^k+\frac1c\bar x^{k-1},
\label{eq:local_convex_combination}\\
x^{k+1}&=\prox_{\lambda_kg}\bigl(\bar x^k-\lambda_kF(x^k)\bigr).
\label{eq:local_prox_update}
\end{align}
\end{algorithm}
As observed in \cite{malitsky2020golden}, the value of $\lambda_0 > 0$ should be
neither too large nor too small, since it can affect the subsequent iterates. A natural way to select $\lambda_0$ is to set
$\lambda_0 = \frac{\|x^1 - x^0\|}{\|F(x^1) - F(x^0)\|}$, where $x^0$ is a small perturbation of $x^1$. Nonetheless, the goal is to choose $x^0$, $x^1$ such that $F(x^0)\neq F(x^1)$. If $F\equiv0$, none of the above choices works, so one may proceed with any $\lambda_0>0$. From the step size rule \eqref{eq:local_step_size}, we note the two key facts: $0<\theta_k:=\frac{\lambda_k}{\lambda_{k-1}}\leq\rho$, and 
\begin{equation}
\lambda_k\norm{F(x^k)-F(x^{k-1})}
\leq\sigma\norm{x^k-x^{k-1}}, \quad k\ge1.
\label{eq:local_key_inequality}
\end{equation}

\subsection{Choice of parameters}
\label{subsec:adaptive_algorithm_local_case-parameters}
In order to select the values of \(c\), \(\rho\), and \(\sigma\) in Algorithm \ref{Alg:general},
we introduce an auxiliary parameter \(\eta>0\), which is used only in the
convergence analysis and not in the algorithmic updates.
For \(1<c<3\), first select $\eta$ from
\begin{equation}
\max\{0,\eta_-\}<\eta<\eta_+,
\label{eq:select_eta_from_the_range}
\end{equation}
where \(\eta_\pm:=\frac{c-1\pm\sqrt{3(c+1)(3-c)}}{2c}\). Since \(1<c<3\), we have \(\eta_+>0\) and \(\eta_+>\eta_-\). Hence the interval~\eqref{eq:select_eta_from_the_range} is nonempty. Next, set \(\beta=1+\frac{1}{c}-\eta\) and choose
\begin{equation}
1<\rho<r_+:=\frac{\beta}{c}
+\sqrt{\frac{\beta^2}{c^2}+\beta\eta}.
\label{eq:select_rho_from_the_range}
\end{equation}
First observe that
\((c+3)^2-3(c+1)(3-c)=4c^2>0\). Thus
\(\eta<\eta_+<1+1/c\), and hence $\beta>0$. Furthermore, note that the polynomial \(p(t):=t^2-2\beta t/c-\beta\eta\) has two roots
\(r_\pm=\beta/c\pm\sqrt{\beta^2/c^2+\beta\eta}\),
with \(r_-<0<r_+\). By the choice of \(\eta\) in \eqref{eq:select_eta_from_the_range}, we have
\[
p(1)=1-\beta\left(\eta+\frac2c\right)
=\eta^2-\frac{c-1}{c}\eta+\frac{c^2-2c-2}{c^2}=-(\eta_+-\eta)(\eta-\eta_-)<0.
\]
Since \(p\) is negative between its roots, we obtain \(1<r_+\). Therefore, the interval in~\eqref{eq:select_rho_from_the_range} is nonempty. Finally, for this choice of \(\rho\), take
\begin{equation}
0<\sigma<\sigma_{\max}:=
\frac{\min\left\{\beta\eta,\,
\beta\eta+\frac{2\beta\rho}{c}-\rho^2\right\}}
{2\left(\eta+\frac{\beta}{c^2}\right)}.
\label{eq:select_sigma_from_the_bound}
\end{equation}
For \(1<\rho<r_+\), we have
\(\beta\eta+2\beta\rho/c-\rho^2=(r_+-\rho)(\rho-r_-)>0\).
Together with \(\beta,\eta>0\), this gives \(\sigma_{\max}>0\). 

\medskip
Through a concrete example, we now show how the parameters determined by \eqref{eq:select_eta_from_the_range}--\eqref{eq:select_sigma_from_the_bound} can be selected in practice. We start with \(c = 2\), which gives \(\eta_+ = 1\) and
\(\eta_- = -1/2\).  With the choice
\(\eta=\frac{1}{4}\), one obtains \(\beta=\tfrac54\) and 
\(r_+=\frac{5+3\sqrt5}{8}>\tfrac54\). Now taking \(\rho=\tfrac54\)
gives $\sigma_{\max}=\frac5{18}$. Consequently, any value of the
form $\sigma=(1-\varepsilon)\sigma_{\max}$ with $0<\varepsilon<1$ is admissible. For instance, the choice of $\varepsilon=10^{-2}$ yields \(\sigma=\frac{11}{40}\). This is precisely the example we discussed in Subsection \ref{sec:contributions}.

\subsection{Convergence analysis}
\label{subsec:adaptive_algorithm_local_case-convergence-analysis}
For $z\in\mathbb{H}$, define the Lyapunov energy
\begin{equation*}
\mathcal E_k(z):=
\frac c{c-1}\norm{\bar x^k-z}^2
+\eta\norm{x^k-\bar x^{k-1}}^2
+\sigma\norm{x^k-x^{k-1}}^2, \quad k\ge1.
\end{equation*}
Since \(\beta>0\) and \(\eta>0\), it follows from
\eqref{eq:select_sigma_from_the_bound} that
\(2\sigma<\dfrac{\beta\eta}{\eta+\frac{\beta}{c^2}}<\beta\).
We may therefore set
\begin{equation}
\Delta(\theta):=\frac{q(\theta)-2\sigma D}{\beta-2\sigma},
\label{eq:local_descent_term_general}
\end{equation}
where $q(\theta):=\beta\eta+\frac{2\beta}{c}\theta-\theta^2$ and $D:=\eta+\frac{\beta}{c^2}$. 

\medskip
Next, we have the following lemma.
\begin{lemma}
\label{lem:local_energy}
Suppose that (\ref{ass:operators})--(\ref{ass:solution}) hold. Let \((x^k)\) and
\((\bar x^k)\) be
generated by Algorithm~\ref{Alg:general}. Then, for every
\(z\in\dom g\) and \(k\geq2\),
\begin{equation}
\mathcal E_{k+1}(z)+2\lambda_k\Psi_z(x^k)
+\Delta(\theta_k)\norm{x^k-\bar x^{k-1}}^2
\leq\mathcal E_k(z).
\label{eq:adaptive_descent_ineq}
\end{equation}
\end{lemma}

\begin{proof}
Applying \eqref{eq:prox_inequality} to the proximal step
\eqref{eq:local_prox_update} gives
\begin{equation}
\ip{x^{k+1}-\bar x^k}{z-x^{k+1}}
+\lambda_k\ip{F(x^k)}{z-x^{k+1}}
\geq\lambda_k\bigl(g(x^{k+1})-g(z)\bigr).
\label{eq:proximal_step_k}
\end{equation}
Applying the proximal inequality at the $(k-1)$-th step with \(z:=x^{k+1}\) and multiplying the result by
\(\theta_k=\frac{\lambda_k}{\lambda_{k-1}}\) yields
\begin{equation}
\theta_k\ip{x^k-\bar x^{k-1}}{x^{k+1}-x^k}
+\lambda_k\ip{F(x^{k-1})}{x^{k+1}-x^k}
\geq\lambda_k\bigl(g(x^k)-g(x^{k+1})\bigr).
\label{eq:proximal_step_k-1}
\end{equation}
Adding \eqref{eq:proximal_step_k} and \eqref{eq:proximal_step_k-1}, followed by using monotonicity of $F$ as in~\eqref{eq:ineq_1}, we obtain
\begin{multline}\label{eq:proximal_step_combined_k_and_k-1}
\lambda_k\Psi_z(x^k)
\leq\ip{x^{k+1}-\bar x^k}{z-x^{k+1}}
+\theta_k\ip{x^k-\bar x^{k-1}}{x^{k+1}-x^k}\\
+\lambda_k\ip{F(x^k)-F(x^{k-1})}{x^k-x^{k+1}}.
\end{multline}
By \eqref{eq:local_convex_combination}, we have
\(x^{k+1}-x^k=(x^{k+1}-\bar x^k)-\frac1c(x^k-\bar x^{k-1})\). Thus
\begin{align}
2\theta_k\ip{x^k-\bar x^{k-1}}{x^{k+1}-x^k}
&=2\theta_k\ip{x^{k+1}-\bar x^k}{x^k-\bar x^{k-1}}
-\frac{2\theta_k}{c}\norm{x^k-\bar x^{k-1}}^2.
\label{eq:local_eq_1}
\end{align}
Multiplying \eqref{eq:proximal_step_combined_k_and_k-1} by $2$ and then substituting \eqref{eq:local_eq_1} and \eqref{eq:ineq_2}, we obtain
\begin{multline}
2\lambda_k\Psi_z(x^k)+\frac c{c-1}\norm{\bar x^{k+1}-z}^2
\leq\frac c{c-1}\norm{\bar x^k-z}^2
-\left(1+\frac1c\right)\norm{x^{k+1}-\bar x^k}^2\\
+2\theta_k\ip{x^{k+1}-\bar x^k}{x^k-\bar x^{k-1}}
-\frac{2\theta_k}{c}\norm{x^k-\bar x^{k-1}}^2
+2\lambda_k\ip{F(x^k)-F(x^{k-1})}{x^k-x^{k+1}}.
\label{eq:local_ineq_1}
\end{multline}
Using \eqref{eq:local_key_inequality} and Young's inequality, we have
\begin{align*}
2\lambda_k\ip{F(x^k)-F(x^{k-1})}{x^k-x^{k+1}}
&\leq2\sigma\norm{x^k-x^{k-1}}\norm{x^{k+1}-x^k}\notag\\
&\leq\sigma\norm{x^k-x^{k-1}}^2
+\sigma\norm{x^{k+1}-x^k}^2.
\end{align*}
Substituting this estimate into \eqref{eq:local_ineq_1} and then
adding \(\sigma\norm{x^{k+1}-x^k}^2\) to both sides of the result yields
\begin{multline}
\frac c{c-1}\norm{\bar x^{k+1}-z}^2
+\sigma\norm{x^{k+1}-x^k}^2+2\lambda_k\Psi_z(x^k)
\leq\frac c{c-1}\norm{\bar x^k-z}^2
+\sigma\norm{x^k-x^{k-1}}^2\\
-\left(1+\frac1c\right)\norm{x^{k+1}-\bar x^k}^2
+2\theta_k\ip{x^{k+1}-\bar x^k}{x^k-\bar x^{k-1}}
-\frac{2\theta_k}{c}\norm{x^k-\bar x^{k-1}}^2\\
+2\sigma\norm{x^{k+1}-x^k}^2.
\label{eq:local_ineq_2}
\end{multline}
Again using
\(x^{k+1}-x^k=(x^{k+1}-\bar x^k)-\frac1c(x^k-\bar x^{k-1})\)
to expand the last term in \eqref{eq:local_ineq_2}, we find
\begin{multline}
\frac c{c-1}\norm{\bar x^{k+1}-z}^2
+\sigma\norm{x^{k+1}-x^k}^2+2\lambda_k\Psi_z(x^k)
\leq\frac c{c-1}\norm{\bar x^k-z}^2
+\sigma\norm{x^k-x^{k-1}}^2\\
-\left(1+\frac1c-2\sigma\right)\norm{x^{k+1}-\bar x^k}^2
+2\left(\theta_k-\frac{2\sigma}{c}\right)
\ip{x^{k+1}-\bar x^k}{x^k-\bar x^{k-1}}\\
-\left(\frac{2\theta_k}{c}-\frac{2\sigma}{c^2}\right)
\norm{x^k-\bar x^{k-1}}^2.
\label{eq:local_ineq_2_new}
\end{multline}
Denote \(\tau:=\beta-2\sigma>0\). Applying Young's inequality with weight $\tau$, we have
\begin{align}
2\left(\theta_k-\frac{2\sigma}{c}\right)
\ip{x^{k+1}-\bar x^k}{x^k-\bar x^{k-1}}
&\leq\tau\norm{x^{k+1}-\bar x^k}^2+\frac{(\theta_k-2\sigma/c)^2}{\tau}
\norm{x^k-\bar x^{k-1}}^2.
\label{eq:local_weighted_young}
\end{align}
Substituting \eqref{eq:local_weighted_young} into
\eqref{eq:local_ineq_2_new}, the coefficient of
\(\norm{x^{k+1}-\bar x^k}^2\) becomes
\(\tau-1-\frac1c+2\sigma=-\eta\).
The coefficient of \(\norm{x^k-\bar x^{k-1}}^2\) is
\[
C(\theta_k):=
\frac{(\theta_k-2\sigma/c)^2}{\beta-2\sigma}
-\frac{2\theta_k}{c}+\frac{2\sigma}{c^2}.
\]
To recover the corresponding Fej\'er term in the energy at the $k$-th step, we write this coefficient as
\begin{align*}
C(\theta_k)
=\frac{(\theta_k-2\sigma/c)^2
-(2\theta_k/c-2\sigma/c^2)(\beta-2\sigma)}
{\beta-2\sigma}
&=\frac{\theta_k^2-\frac{2\beta}{c}\theta_k
+\frac{2\sigma\beta}{c^2}}{\beta-2\sigma}\\
&=\eta-\left(\eta-\frac{\theta_k^2-\frac{2\beta}{c}\theta_k
+\frac{2\sigma\beta}{c^2}}{\beta-2\sigma}\right)\\
&=\eta-
\frac{\beta\eta+\frac{2\beta}{c}\theta_k-\theta_k^2
-2\sigma\left(\eta+\frac{\beta}{c^2}\right)}
{\beta-2\sigma}\\
&=\eta-\Delta(\theta_k).
\end{align*}
Combining these results with~\eqref{eq:local_ineq_2_new} proves~\eqref{eq:adaptive_descent_ineq}.
\end{proof}

Before stating the main convergence theorem, note that the concavity of \(q\) reduces the uniform descent requirement in \eqref{eq:adaptive_descent_ineq} to its endpoint values at \(0\) and \(\rho\). We first use this observation in Theorem~\ref{thm:local_convergence} to prove the boundedness of the iterates. Only then do we apply the following lemma to bound the step sizes away from zero.

\begin{lemma}
\label{lem:local_step_size_bdd_below}
Suppose that \(F:\overline{\dom g}\to\mathbb{H}\) is locally Lipschitz. If the sequence \((x^k)\) generated by Algorithm~\ref{Alg:general} is bounded, then $(\lambda_k)$ is bounded away from zero.
\end{lemma}
\begin{proof}
Since $(x^k)$ is bounded, the set
\(Q=\overline{\{x^k:k\geq0\}}\subseteq\overline{\dom g}\) is compact. Thus, by the definition of local Lipschitz continuity of \(F\) on $\overline{\dom g}$, there exists a constant \(L_Q>0\) such that
\(\norm{F(x^k)-F(x^{k-1})}\leq L_Q\norm{x^k-x^{k-1}}\) for all \(k\geq1\). Now one can observe that the second term in~\eqref{eq:local_step_size} is at least
\(\sigma/L_Q\). Therefore
\[
\lambda_k\geq
\min\left\{\rho\lambda_{k-1},\frac{\sigma}{L_Q}\right\}
\geq
\min\left\{\lambda_{k-1},\frac{\sigma}{L_Q}\right\}.
\]
By induction, we obtain $\lambda_k\geq\bar \lambda
:=\min\left\{\lambda_0,\frac{\sigma}{L_Q}\right\}>0.$
\end{proof}
The following theorem is the main result regarding the convergence of Algorithm~\ref{Alg:general}.
\begin{theorem}
\label{thm:local_convergence}
Suppose that \(F:\overline{\dom g}\to\mathbb{H}\) is locally Lipschitz, and conditions (\ref{ass:operators})--(\ref{ass:solution}) hold.
Then both sequences \((x^k)\) and \((\bar x^k)\) generated by
Algorithm~\ref{Alg:general} converge to the same solution of~\eqref{eq:monotone_vi}.
\end{theorem}

\begin{proof}
We have already shown that \(\beta>2\sigma\).
Further, by \eqref{eq:select_sigma_from_the_bound}, we have \(\min\{q(0),q(\rho)\}>2\sigma D\).
Since \(q\) is a concave polynomial, 
\(q(\theta)\geq\min\{q(0),q(\rho)\}\) for all
\(0<\theta\leq\rho\). Therefore, it follows that
\begin{equation}
\Delta(\theta)\geq\delta_*
:=\frac{\min\{q(0),q(\rho)\}-2\sigma D}{\beta-2\sigma}>0.
\label{eq:uniform_delta}
\end{equation}
For a solution \(z^*\in\mathcal{S}\),
\(\Psi_{z^*}(x^k)\geq0\). Hence, Lemma~\ref{lem:local_energy} yields
\begin{equation}
\mathcal E_{k+1}(z^*)+\delta_*\norm{x^k-\bar x^{k-1}}^2
\leq\mathcal E_k(z^*).
\label{eq:local_Lyapunov-energy_monotone}
\end{equation}
Thus \(\lim_{k\to\infty}\mathcal E_k(z^*)\) exists. As in the proof of Theorem~\ref{thm:fixed}, a similar argument shows that both $(x^k)$ and $(\bar x^k)$ are bounded and
\[
\lim_{k\to\infty}(x^k-\bar x^{k-1})=\lim_{k\to\infty}
(x^k-\bar x^k)=\lim_{k\to\infty}(x^{k+1}-x^k)=0.
\]
Take any subsequence $(x^{k_j})$ such that
\(x^{k_j}\to\tilde x\). Since \(\tilde x\in\overline{\dom g}\), by continuity of \(F\), we have
\(F(x^{k_j})\to F(\tilde x)\). By the proximal optimality condition, we have
\[
\frac{\bar x^{k_j}-x^{k_j+1}}{\lambda_{k_j}}-F(x^{k_j})
\in\partial g(x^{k_j+1}).
\]
Note from Lemma~\ref{lem:local_step_size_bdd_below} that \(\lambda_k\geq\bar \lambda>0\). Passing to the limit as $j\to\infty$ and using the closedness of the graph of
\(\partial g\), we obtain \(-F(\tilde x)\in\partial g(\tilde x)\).
In particular, \(\tilde x\in\dom g\), and every cluster point is a solution. Finally, by \eqref{eq:local_Lyapunov-energy_monotone}, \(\lim_{k\to\infty}\norm{\bar x^k-z^*}\) exists for every
\(z^*\in\mathcal S\).
Therefore, Lemma~\ref{lem:useful_1}, together with
\(\lim_{k\to\infty}(x^k-\bar x^k)=0\), completes the proof.
\end{proof}

\begin{theorem}
\label{thm:local_ergodic_gap_rate_convergence}
Under the assumptions of Theorem~\ref{thm:local_convergence}, let
\(U\subset\dom g\) be nonempty and compact, and define
\[
\Lambda_N:=\sum_{k=2}^N\lambda_k,
\qquad
X_N:=\frac1{\Lambda_N}\sum_{k=2}^N\lambda_kx^k.
\]
Then, for every \(N\geq2\),
\begin{align}
\mathcal G_U(X_N)
&\leq\frac{M_U}
{2(N-1)\min\{\lambda_0,\sigma/L_Q\}},
\label{eq:local_adaptive_ergodic_rate}
\end{align}
where $M_U:=\sup_{z\in U}
\left\{\frac c{c-1}\norm{\bar x^2-z}^2
+(\eta+\sigma)\norm{x^2-x^1}^2\right\},$ and $L_Q>0$ is a Lipschitz constant of $F$ on the compact set
$Q$ in Lemma~\ref{lem:local_step_size_bdd_below}. In particular, if we choose $\lambda_0$ such that $\lambda_0\geq\sigma/L_Q$, then 
\begin{equation}
\mathcal G_U(X_N)\leq\frac{L_Q M_U}{2\sigma(N-1)}.
\label{eq:adaptive_gap_derived}
\end{equation}
\end{theorem} 
\begin{proof}
Fix $z\in U$. Summing \eqref{eq:adaptive_descent_ineq} for $k=2,\dots,N$ and then dropping the nonnegative energy terms together with $\bar x^1=x^1$ yields
\[
2\sum_{k=2}^N\lambda_k\Psi_z(x^k)\leq\mathcal E_2(z):=\frac c{c-1}\norm{\bar x^2-z}^2
+(\eta+\sigma)\norm{x^2-x^1}^2.
\]
By convexity of \(\Psi_z\), the left-hand side is at least
\(2\Lambda_N\Psi_z(X_N)\). Now taking the supremum over \(z\in U\) and using \(\Lambda_N\geq\min\{\lambda_0,\sigma/L_Q\}(N-1)\) gives
\eqref{eq:local_adaptive_ergodic_rate}. 

For \eqref{eq:adaptive_gap_derived}, note that $\sigma<\beta/2<1$. Hence, if $F(x^1)\neq F(x^0)$, 
\[
\lambda_0=\frac{\norm{x^1-x^0}}{\norm{F(x^1)-F(x^0)}}
\geq\frac1{L_Q}>\frac{\sigma}{L_Q},
\]
from which the claim follows via \eqref{eq:local_adaptive_ergodic_rate}.
\end{proof}

\subsection{Linear convergence under an error bound}
\label{sec:linear_rate}
We next show that under an error bound condition, AE-GRAAL exhibits an
R-linear rate of convergence. To achieve this, given \(t>0\) and \(x\in\dom g\), define the natural residual
\begin{equation}
r_t(x):=x-\prox_{tg}\bigl(x-tF(x)\bigr).
\label{eq:local_natural_residual}
\end{equation}
By the proximal optimality condition, \(r_t(x)=0\) if and only if
\(x\in\mathcal S\). We assume that, for some fixed \(t>0\),
there exist \(\kappa,\varepsilon_0>0\) such that
\begin{equation}
\dist(x,\mathcal S)\leq\kappa\norm{r_t(x)}
\quad\text{whenever }x\in\dom g\text{ and }
\norm{r_t(x)}\leq\varepsilon_0.
\label{eq:local_error_bound}
\end{equation}
This error-bound condition is commonly used to establish
linear convergence of methods for variational inequalities; see, for example
\cite{pang1997error,tseng1995linear,solodov2003convergence}. The author~\cite[Theorem~3]{malitsky2020golden} used this natural residual to prove an R-linear rate for aGRAAL. A similar result can be shown for Algorithm~\ref{Alg:general}. To this end, we first record
the standard residual comparison, which allows us to decrease \(t\)
while adjusting the error-bound constants; see also
\cite{FacchineiPang2003}.

\begin{lemma}
\label{lem:residual_equivalence}
For \(x\in\dom g\) and \(0<s\leq t\),
\begin{equation}
\frac{s}{t}\norm{r_t(x)}\leq\norm{r_s(x)}\leq\norm{r_t(x)}.
\label{eq:residual_equivalence}
\end{equation}
\end{lemma}

\begin{proof}
By the proximal optimality condition, \eqref{eq:local_natural_residual} yields
\[
\frac{r_s(x)}s-F(x)\in\partial g\bigl(x-r_s(x)\bigr)
\quad\text{and}\quad
\frac{r_t(x)}t-F(x)\in\partial g\bigl(x-r_t(x)\bigr).
\]
Now by monotonicity of \(\partial g\), we obtain
\[
t\norm{r_s(x)}^2+s\norm{r_t(x)}^2
\leq(s+t)\ip{r_s(x)}{r_t(x)}
\leq(s+t)\norm{r_s(x)}\norm{r_t(x)}.
\]
If \(r_t(x)=0\), this inequality also gives \(r_s(x)=0\).
Otherwise, division by \(\norm{r_t(x)}^2\) shows that
\(\norm{r_s(x)}/\norm{r_t(x)}\) lies between \(s/t\) and \(1\),
which proves \eqref{eq:residual_equivalence}.
\end{proof}

\begin{theorem}
\label{thm:loca_linear_rate}
Under the assumptions of Theorem~\ref{thm:local_convergence},
let \((x^k)\) and \((\bar x^k)\) be generated by
Algorithm~\ref{Alg:general}. If \eqref{eq:local_error_bound}
holds, then both sequences converge R-linearly to the same solution of~\eqref{eq:monotone_vi}.
\end{theorem}

\begin{proof}
By Theorem~\ref{thm:local_convergence} and Lemma~\ref{lem:local_step_size_bdd_below}, there exists
\(\bar \lambda>0\) such that
\(\lambda_k\geq\bar \lambda\) for all \(k\geq1\).
By Lemma~\ref{lem:residual_equivalence}, we may decrease \(t\)
and adjust the error-bound constants so that
\(0<t\leq\bar \lambda\).
Therefore, Lemma~\ref{lem:residual_equivalence} and nonexpansiveness of the proximal mapping give
\begin{align}
\norm{r_t(x^k)}
&\leq\norm{r_{\lambda_k}(x^k)}
\leq\norm{x^{k+1}-x^k}+\norm{x^k-\bar x^k}\leq\norm{x^{k+1}-\bar x^k}
+\frac2c\norm{x^k-\bar x^{k-1}}.
\label{eq:local_residual_estimate}
\end{align}
Let $x^*$ be the common limit supplied by Theorem~\ref{thm:local_convergence}. Since \(r_t\) is continuous and \(r_t(x^*)=0\), we have
\(\lim_{k\to\infty}r_t(x^k)=0\). Thus, the error bound applies for every \(k\geq k_0\), for some $k_0\in\mathbb N$. 
Since~\eqref{eq:local_Lyapunov-energy_monotone} holds for every
\(z\in\mathcal S\), it follows that
\begin{equation}
\mathcal E_{k+1}^{\mathcal S}
+\delta_*\norm{x^k-\bar x^{k-1}}^2
\leq\mathcal E_k^{\mathcal S},
\qquad k\geq2,
\label{eq:local_set_descent}
\end{equation}
where
\begin{equation}
\mathcal E_k^{\mathcal S}
:=\inf_{z\in\mathcal S}\mathcal E_k(z)
=\frac{c}{c-1}\operatorname{dist}(\bar x^k,\mathcal S)^2
+\eta\norm{x^k-\bar x^{k-1}}^2
+\sigma\norm{x^k-x^{k-1}}^2.
\label{eq:local_energy_1}
\end{equation}
Thus, the sequence $\{\mathcal E_k^{\mathcal S}\}_{k\ge2}$ is nonincreasing. To relate this decrease to the energy itself, we use the identities
\begin{align*}
\bar x^{k+1}-x^k
&=\frac{c-1}{c}(x^{k+1}-\bar x^k)
-\frac1c(x^k-\bar x^{k-1}),\\
x^{k+1}-x^k
&=(x^{k+1}-\bar x^k)
-\frac1c(x^k-\bar x^{k-1}),
\end{align*}
which follow from \eqref{eq:local_convex_combination}.
The first identity and the error bound give, for \(k\geq k_0\)
\begin{align*}
\operatorname{dist}(\bar x^{k+1},\mathcal S)
&\leq\norm{\bar x^{k+1}-x^k}
+\kappa\norm{r_t(x^k)}\leq\left(\kappa+\frac{c-1}{c}\right)
\norm{x^{k+1}-\bar x^k}
+\frac{2\kappa+1}{c}\norm{x^k-\bar x^{k-1}}.
\end{align*}
Substituting these estimates into
\eqref{eq:local_energy_1} and applying
Cauchy--Schwarz, we obtain 
\begin{equation}
\mathcal E_{k+1}^{\mathcal S}
\leq C\left(
\norm{x^{k+1}-\bar x^k}^2
+\norm{x^k-\bar x^{k-1}}^2\right),
\qquad k\geq k_0,
\label{eq:local_energy_2}
\end{equation}
where
\[
C=\frac{c}{c-1}\left[
\left(\kappa+\frac{c-1}{c}\right)^2
+\left(\frac{2\kappa+1}{c}\right)^2\right]
+\eta+\sigma\left(1+\frac1{c^2}\right).
\]
Adding \eqref{eq:local_set_descent} at \(k\) and \(k+1\)
bounds the two squared norms in
\eqref{eq:local_energy_2}. Hence
\[
\mathcal E_{k+2}^{\mathcal S}
\leq\mathcal E_{k+1}^{\mathcal S}
\leq\frac{C}{\delta_*}
\left(\mathcal E_k^{\mathcal S}
-\mathcal E_{k+2}^{\mathcal S}\right),
\]
and rearranging gives, for $k\ge k_0$
\begin{equation*}
\mathcal E_{k+2}^{\mathcal S}
\leq\omega\mathcal E_k^{\mathcal S},
\quad
\omega:=\frac{C}{C+\delta_*}\in(0,1).
\end{equation*}
Let \(\gamma=\omega^{1/4}\). Then iterating the above inequality along the even and odd indices,
and using monotonicity of the energy, yields
\[
\sqrt{\mathcal E_{k+1}^{\mathcal S}}
\leq\sqrt{\mathcal E_{k_0}^{\mathcal S}}\,
\gamma^{k-k_0}.
\]
By the definition of the energy and the identity
\(\bar x^{k+1}-\bar x^k
=(c-1)(x^{k+1}-\bar x^k)/c\), there is a constant \(M>0\)
such that
\[
\norm{x^{k+1}-x^k}
+\norm{\bar x^{k+1}-\bar x^k}
\leq M\gamma^{k-k_0},
\qquad k\geq k_0.
\]
Summing from \(k\) to infinity and using the convergence of
both sequences to \(x^*\), gives
\begin{equation*}
\norm{x^k-x^*}+\norm{\bar x^k-x^*}
\leq\frac{M}{1-\gamma}\gamma^{k-k_0},
\qquad k\geq k_0.
\end{equation*}
This completes the proof.
\end{proof}

\begin{remark}
In \cite[Section~2.3]{malitsky2020golden}, the residual is evaluated
at \(\bar x^k\). Hence, the resulting estimate contains the
factor \(1+\lambda_kL\), which is bounded using the upper
step size cap. In \eqref{eq:local_residual_estimate}, evaluating
the residual at \(x^k\) avoids this factor. The positive lower
bound on \(\lambda_k\) and the positive constant
\(\delta_*\) are sufficient for the argument above. Thus, the
R-linear rate requires no explicit cap or further change to the step size rule of Algorithm~\ref{Alg:general}.
\end{remark}
\section{An application to structured variational inequalities}
\label{sec:structured_vi_applications}
We now apply the proposed adaptive method (Algorithm~\ref{Alg:general}) to problems involving a sum of
several convex functions. These types of problems are relevant when multiple constraints or several nonsmooth penalties are present in the formulation of \eqref{eq:monotone_vi}.  Consider
\begin{equation}
\text{find }x^*\in \mathbb H\text{ such that}\quad
\ip{F(x^*)}{x-x^*}
+\sum_{i=1}^m\bigl(g_i(x)-g_i(x^*)\bigr)\geq0
\quad\forall x\in \mathbb H,
\label{eq:vi_application}
\end{equation}
where $m\ge2$, each \(g_i:\mathbb H\to(-\infty,+\infty]\) is proper,
lsc, and convex, and \(F:\overline{\dom g_1}\to\mathbb H\) is
monotone and locally Lipschitz. Assume that \eqref{eq:vi_application} has a solution and that the sum rule holds: that is, for any $x\in\bigcap_{i=1}^m\dom g_i$
\begin{equation}
\partial\left(\sum_{i=1}^m g_i\right)(x)
=\sum_{i=1}^m\partial g_i(x).
\label{eq:sum_rule}
\end{equation}
Since \(\mathbb H\) is finite-dimensional, a sufficient condition
is \(\bigcap_{i=1}^m\operatorname{ri}(\dom g_i)\neq\varnothing\),
where ``\(\operatorname{ri}\)'' denotes the \emph{relative interior} of a set; see \cite[Definition 6.9]{BauschkeCombettes2017}. To solve \eqref{eq:vi_application}, one could apply~Algorithm~\ref{Alg:general} replacing $g$ with \(g:=\sum_i g_i\). However, it would then require computing the proximal mapping of the sum, which can be difficult even when the
proximal mappings of the individual terms are available. Therefore, we derive a formulation that keeps these proximal computations separate. By the optimality condition for~\eqref{eq:vi_application} and the sum rule~\eqref{eq:sum_rule}, we obtain
\begin{equation*}
0\in F(x^*)+\sum_{i=1}^m\partial g_i(x^*).
\end{equation*}
Thus, there exist \(y_i^*\in\partial g_i(x^*)\),
\(i=2,\ldots,m\), such that
\(-F(x^*)-\sum_{i=2}^m y_i^*\in\partial g_1(x^*)\). Furthermore, observe that \(y_i^*\in\partial g_i(x^*)\) if and only if
\(x^*\in\partial g_i^*(y_i^*)\). Now, we have
\begin{equation}
\begin{cases}
0\in F(x^*)+\partial g_1(x^*)
+\displaystyle\sum_{i=2}^m y_i^*,\\[1mm]
0\in\partial g_i^*(y_i^*)-x^*,
\qquad i=2,\ldots,m,
\end{cases}
\label{eq:optimality}
\end{equation}
where \(g_i^*\) is the Fenchel conjugate \cite[Definition 13.1]{BauschkeCombettes2017} of \(g_i\). Conversely, the same
subdifferential identity shows that any solution of~\eqref{eq:optimality} gives a solution of~\eqref{eq:vi_application}. To write these conditions as a single inclusion, equip
\(\mathbb K=\mathbb H^m\) with the usual
product inner product. Then for \(z=(x,y_2,\ldots,y_m)\), define
\begin{align*}
G(z)&:=g_1(x)+\sum_{i=2}^m g_i^*(y_i),\\
B(z)&:=\left(F(x)+\sum_{i=2}^m y_i,-x,\ldots,-x\right).
\end{align*}
By separability of \(G\), the primal-dual system
\eqref{eq:optimality} is equivalent to
\begin{equation}
0\in B(z^*)+\partial G(z^*).
\label{eq:equivalent_form}
\end{equation}
Further, note that for \(y=(y_2,\ldots,y_m)\) and
\(Kx=(x,\ldots,x)\in\mathbb H^{m-1}\), we have
\(K^*y=\sum_{i=2}^m y_i\) and
\begin{equation*}
B(z)=(F(x),0)+Mz,
\qquad
M:=\begin{pmatrix}0&K^*\\-K&0\end{pmatrix},
\qquad M^*=-M.
\end{equation*}
This is the standard monotone-plus-skew formulation of
\cite[Section~2.3]{BricenoAriasCombettes2011}, where further extension with several terms and a monotone Lipschitz operator is developed in
\cite[Section~3]{combettes2012primal}. Here we use this formulation
to apply Algorithm~\ref{Alg:general} under local Lipschitz continuity of $F$. Since the proximal map of \(G\) is separable, applying Algorithm~\ref{Alg:general} to \eqref{eq:equivalent_form} yields the following implementation.
\begin{algorithm}[H]
\caption{Adaptive reformulated golden ratio algorithm (AR-GRAAL) for \eqref{eq:vi_application}.}\label{Alg:application}
Let \(1<c<3\), and choose \(\rho,\sigma\) according to
\eqref{eq:select_rho_from_the_range}--\eqref{eq:select_sigma_from_the_bound}, respectively.
Let \(z^0,z^1\in\dom G\), \(\lambda_0>0\), and set
\(\bar z^0=z^1\), where
\(z^k=(x^k,y_2^k,\ldots,y_m^k)\) and
\(\bar z^k=(\bar x^k,\bar y_2^k,\ldots,\bar y_m^k)\).
For \(k\geq1\), compute
\begin{align*}
\lambda_k&=\min\left\{\rho\lambda_{k-1},
\frac{\sigma\norm{z^k-z^{k-1}}}{\norm{B(z^k)-B(z^{k-1})}}\right\},\\
\bar x^k&=\frac{c-1}{c}x^k+\frac1c\bar x^{k-1},\notag\\
\bar y_i^k&=\frac{c-1}{c}y_i^k+\frac1c\bar y_i^{k-1},\\
x^{k+1}&=\prox_{\lambda_k g_1}\left(
\bar x^k-\lambda_k\left(F(x^k)+\sum_{i=2}^m y_i^k\right)\right),\\
y_i^{k+1}&=\prox_{\lambda_k g_i^*}
\bigl(\bar y_i^k+\lambda_k x^k\bigr),\qquad i=2,\ldots,m.
\end{align*}
\end{algorithm}
In practice, knowing $\prox_{g_i}$ is enough, as we can always apply Moreau's identity~\cite[Theorem 14.3 (ii)]{BauschkeCombettes2017} to compute the proximal mapping of the conjugates. Further, note that 
\begin{align*}
\norm{z^k-z^{k-1}}^2&:=\norm{x^k-x^{k-1}}^2+\sum_{i=2}^m\norm{y_i^k-y_i^{k-1}}^2,\notag\\
\norm{B(z^k)-B(z^{k-1})}^2&:=\left\|F(x^k)-F(x^{k-1})+\sum_{i=2}^m (y_i^k-y_i^{k-1})\right\|^2
+(m-1)\norm{x^k-x^{k-1}}^2.
\end{align*}

\begin{corollary}
Suppose that \(F:\overline{\dom g_1}\to\mathbb{H}\) is monotone and locally Lipschitz, and for $i=1,2,\ldots,m$ that \(g_i:\mathbb H\to(-\infty,+\infty]\) is proper
lsc and convex. Assume that \eqref{eq:vi_application} has a solution and that the sum rule~\eqref{eq:sum_rule} holds. Then the sequences \((z^k)\) and \((\bar z^k)\) generated by Algorithm~\ref{Alg:application} converge to the same
solution of \eqref{eq:equivalent_form}.
In particular, \((x^k)\) converges to a solution of~\eqref{eq:vi_application}.
\end{corollary}
\begin{proof}
The function \(G\) is proper, lsc, and convex, \(B\) is defined on
\(\overline{\dom G}\), and the lifted solution set is nonempty by
\eqref{eq:equivalent_form}.
For \(z,z'\in\overline{\dom G}\), cancellation of the coupling terms gives
\[
\ip{B(z)-B(z')}{z-z'}=\ip{F(x)-F(x')}{x-x'}\geq0.
\]
Thus \(B\) is monotone. Now if $F$ is $L$-Lipschitz on a compact set, then the triangle inequality and Cauchy--Schwarz inequality give
\begin{align*}
\norm{B(z)-B(z')}
&\leq L\norm{x-x'}+
\left(\left\|\sum_{i=2}^m(y_i-y'_i)\right\|^2
+(m-1)\norm{x-x'}^2\right)^{1/2}\\
&\leq\bigl(L+\sqrt{m-1}\bigr)\norm{z-z'}.
\end{align*}
Hence \(B\) is locally Lipschitz on \(\overline{\dom G}\),
and all the assumptions of Theorem~\ref{thm:local_convergence}
hold. Therefore, Theorem \ref{thm:local_convergence} gives \(z^k,\bar z^k\to z^*:=(x^*,y_2^*,\ldots,y_m^*)\). In particular,
\(x^k\to x^*\), and~\eqref{eq:equivalent_form}
shows that \(x^*\) is a solution of \eqref{eq:vi_application}.
\end{proof}

\section{Numerical illustrations}\label{sec:numerical_results}

In this section, we compare two instances of Algorithm~\ref{Alg:general}, namely
AE-GRAAL-A and AE-GRAAL-B, against aGRAAL \cite{malitsky2020golden} on three benchmark problems: the
Nash--Cournot problem, sparse logistic regression, and the cubic
complementarity problem. Both instances use
\[
\begin{array}{c|cccc}
 &c&\eta&\rho&\sigma\\ \hline
\text{AE-GRAAL-A}&2&1/4&5/4&11/40\\
\text{AE-GRAAL-B}&2&1/3&7/6&77/250
\end{array}
\]
The parameter selection of AE-GRAAL-A is verified in~Subsection~\ref{subsec:adaptive_algorithm_local_case-parameters}. Taking $c=2$ and $\eta=1/3$, and setting
 $\rho=\beta=3/2-\eta=7/6$, makes the two terms in the numerator
of \eqref{eq:select_sigma_from_the_bound} equal, which gives
 $\sigma_{\max}=14/45$. Now choosing
$\sigma=(1-\varepsilon)\sigma_{\max}$ with $\varepsilon=10^{-2}$, gives the selected parameters for AE-GRAAL-B.
These parameter choices illustrate two different parameter tradeoffs: AE-GRAAL-A lets its step sizes grow faster
through a larger $\rho$, while AE-GRAAL-B permits a relatively larger
coefficient $\sigma$ in the local step size calculation. It is worth mentioning that these choices are illustrative and do not automatically guarantee better performance over other parameter choices.
For aGRAAL, we follow the numerical implementation of
\cite{malitsky2020golden} and take $\phi=3/2$,
$\rho=1/\phi+1/\phi^2=10/9$, $\vartheta_0=1$, and
$\lambda_{\max}=10^6$.

To study parameter sensitivity in AE-GRAAL, we vary $c$ on the Nash--Cournot and cubic complementarity problems, using the grid
 $c\in\{1.1,1.2,\ldots,2.9\}$. At each grid point we choose
\[
\eta=
\begin{cases}
\dfrac{c-1}{4c}, & \text{Nash--Cournot problem},\\[1ex]
\dfrac{c-1}{2c}, & \text{Cubic complementarity problem},
\end{cases}
\]
and then select $\rho$ from \eqref{eq:select_rho_from_the_range} as the midpoint $\rho=\frac{1+r_+}{2}.$ For this value of $\rho$, we select $\sigma$ from \eqref{eq:select_sigma_from_the_bound} as $\sigma=(1-\varepsilon)\sigma_{\max},$ with $\varepsilon=10^{-2}$. The midpoint avoids choosing $\rho$ too close to its upper bound.
To see why this matters, write
$q(\theta)=\beta\eta+2\beta \theta/c-\theta^2$ and
$D=\eta+\beta/c^2$. For fixed $c$ and $\eta$, we have
$q(r_{+})=0$. Hence $\sigma_{\max}$ can be extremely small if we choose $\rho$ closer to $r_{+}$.
As a result, the choice of $\sigma$ can restrict the step sizes to the second term in \eqref{eq:local_step_size}. 

For each scenario in the Nash--Cournot problem, we use $10$ realisations, and $5$ for the cubic complementarity problem. For both problems, we also compare the step sizes $\lambda_k$ with $1/L_k$, where
\[
L_k=\frac{\|F(x^{k})-F(x^{k-1})\|}
         {\|x^{k}-x^{k-1}\|}
\]
is an estimate of the local Lipschitz constant of $F$. This comparison uses the first
realisation (seed $0$) and shows at most $1000$ updates. The residual plots demonstrate individual runs, and for the sensitivity study of the parameters $c\in\{1.1,1.2,\ldots,2.9\}\subset(1,3)$, we plot median iteration counts across realisations. All numerical experiments were performed using Python 3.14.6 on a 64-bit Windows 11 Pro computer equipped with an Intel Core Ultra 7 266V processor (8 cores) and 16 GB of RAM.

\subsection{Nash--Cournot equilibrium}
We consider a Nash--Cournot model in which $n$ firms compete by
supplying a homogeneous product. Firm $i$ chooses $q_i\geq0$ to
maximise its profit $q_i p(Q)-f_i(q_i)$, where
$Q=\sum_{i=1}^n q_i$ is the total output. Following
\cite{malitsky2020golden}, we use the inverse demand and cost functions
\[
p(Q)=5000^{1/\gamma}Q^{-1/\gamma}
\quad\text{and}\quad
f_i(q_i)=a_iq_i+\frac{\beta_i}{\beta_i+1}
\ell_i^{1/\beta_i}q_i^{1+1/\beta_i},
\]
with $\gamma>1$ and $a_i, \ell_i, \beta_i>0$. We modify the feasible set by setting  $C_{\varepsilon_1}:=[\varepsilon_1,\infty)^n$ with $\varepsilon_1>0$. Therefore, the experiments concern a game with the lower output constraint $q_i\geq\varepsilon_1$, which can 

\begin{figure}[htbp]
  \centering
  \includegraphics[width=0.31\linewidth]{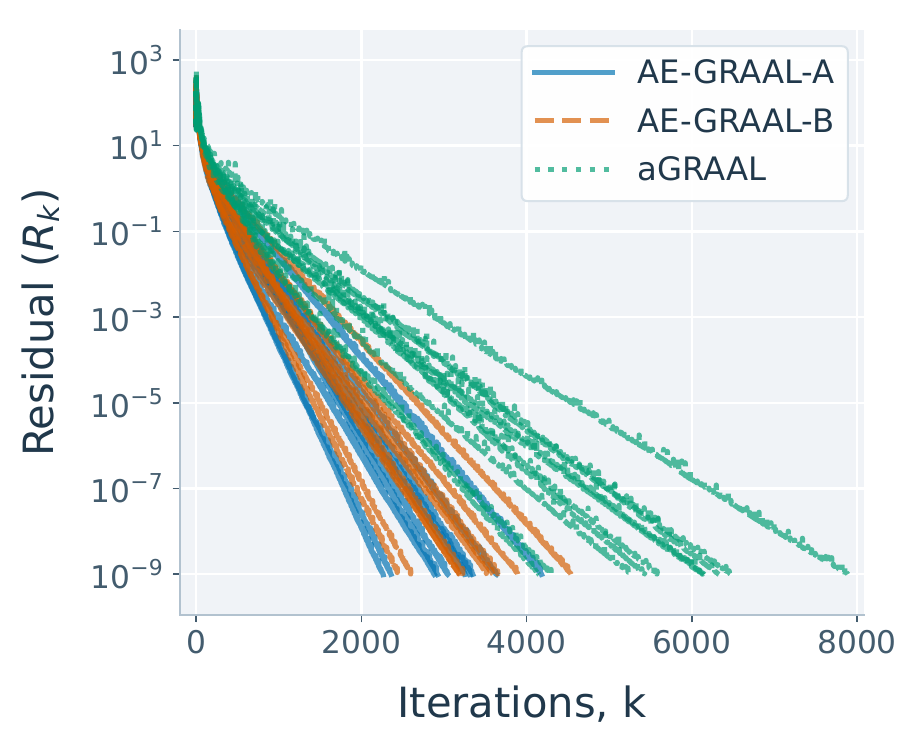}%
  \includegraphics[width=0.31\linewidth]{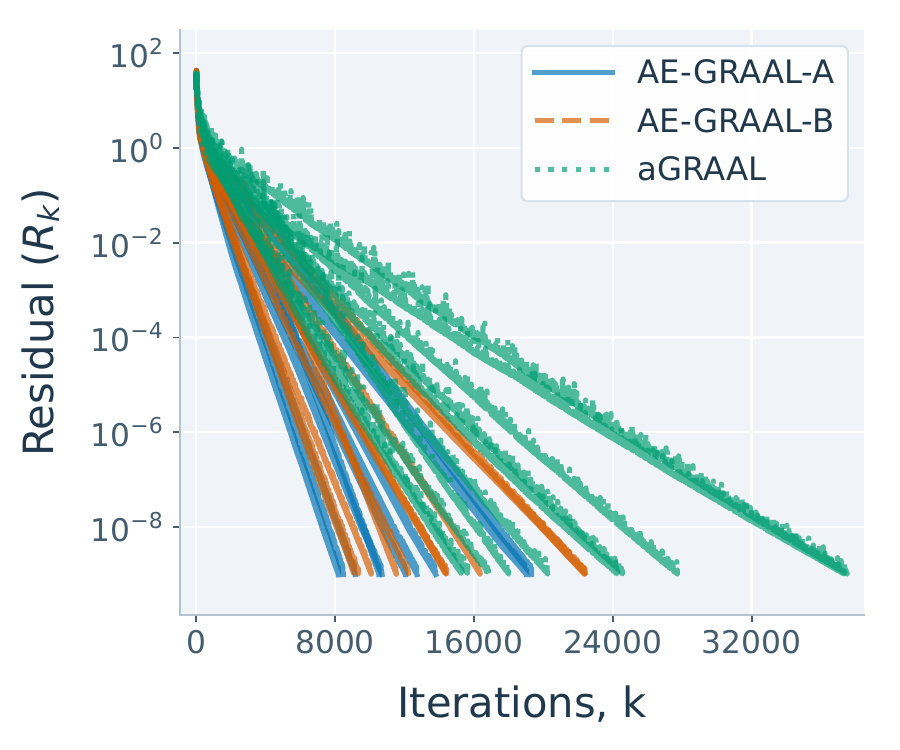}
    \includegraphics[width=0.31\linewidth]{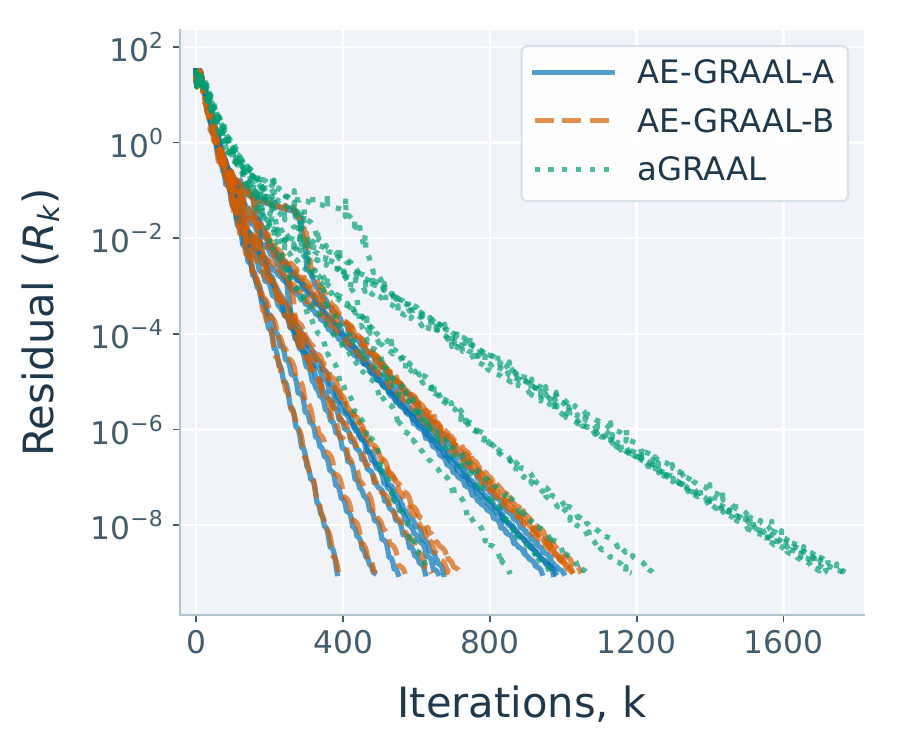}
  \caption{Residual ($R_k$) over iterations for the Nash--Cournot equilibrium problem:
    Scenario~1 (left), Scenario~2 (middle), and Scenario~3 (right).}
  \label{fig:nash_experiment}
\end{figure}
\noindent
have different equilibria from the game with $q_i\geq0$. This restriction addresses two boundary issues in the original benchmark. First, when $\beta_i > 1$, the marginal-cost term $(\ell_iq_i)^{1/\beta_i}$ in $F_i$ is not locally Lipschitz at $q_i=0$. Second, the inverse demand curve is undefined at $Q = 0$. At an equilibrium with positive total output, the first-order
conditions give the VI
\[
\text{find }q^*\in C_{\varepsilon_1}
\quad\text{such that}\quad
\ip{F(q^*)}{q-q^*}\geq0
\quad\forall q\in C_{\varepsilon_1},
\]
where $F_i(q)=a_i+(\ell_iq_i)^{1/\beta_i}-p(Q)-q_ip'(Q).$ Thus, in comparison with \eqref{eq:monotone_vi}, $g=\delta_{C_{\varepsilon_1}}$, and its proximal map is $[\prox_{\lambda g}(v)]_i =\max\{v_i,\varepsilon_1\}.$ For $\gamma>1$, the operator $F$ is monotone on
$C_{\varepsilon_1}$. Furthermore, every feasible point satisfies
$q_i\geq\varepsilon_1$ and $Q\geq n\varepsilon_1$. Moreover, $F$ is
continuously differentiable, and is locally Lipschitz on $\overline{\dom g}=C_{\varepsilon_1}$.
\begin{figure}[htbp]
  \centering
  \includegraphics[width=0.32\linewidth]{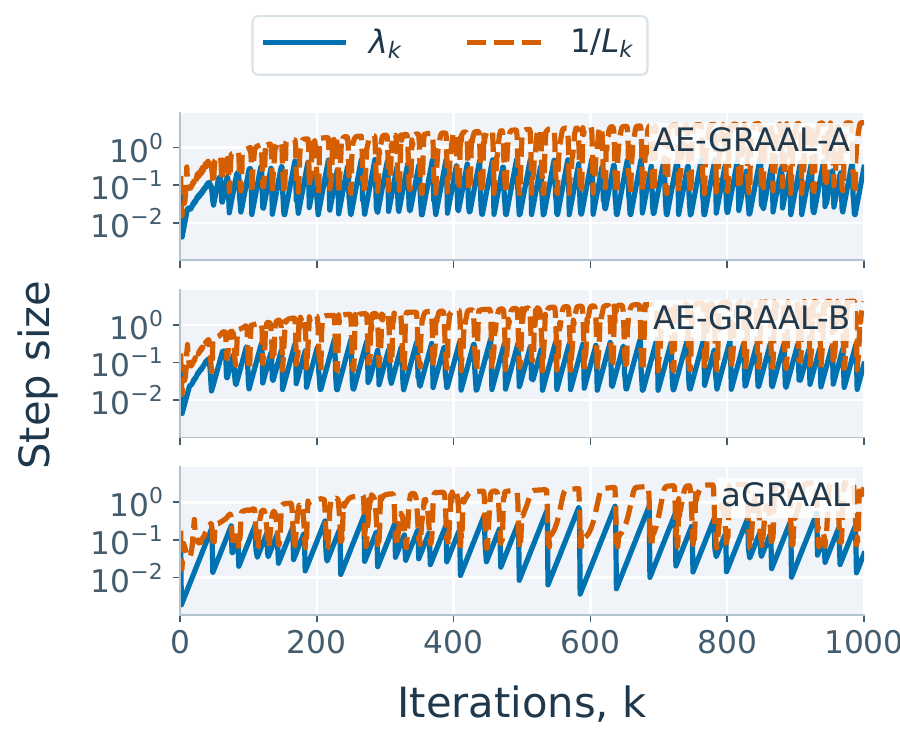}%
  \hspace{1.1cm}
  \includegraphics[width=0.32\linewidth]{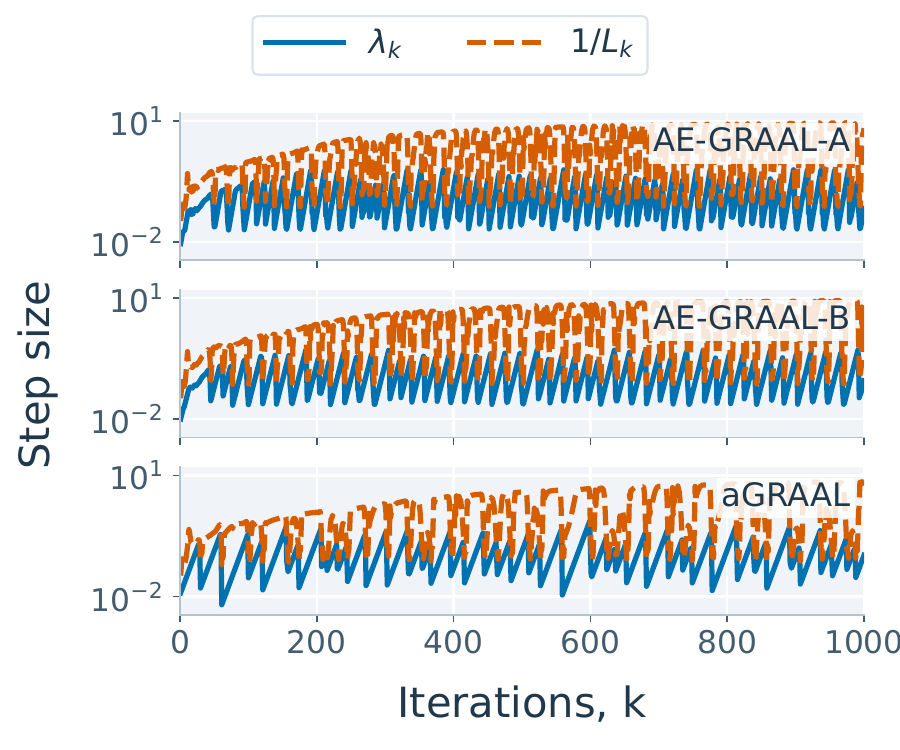}\\
    \includegraphics[width=0.32\linewidth]{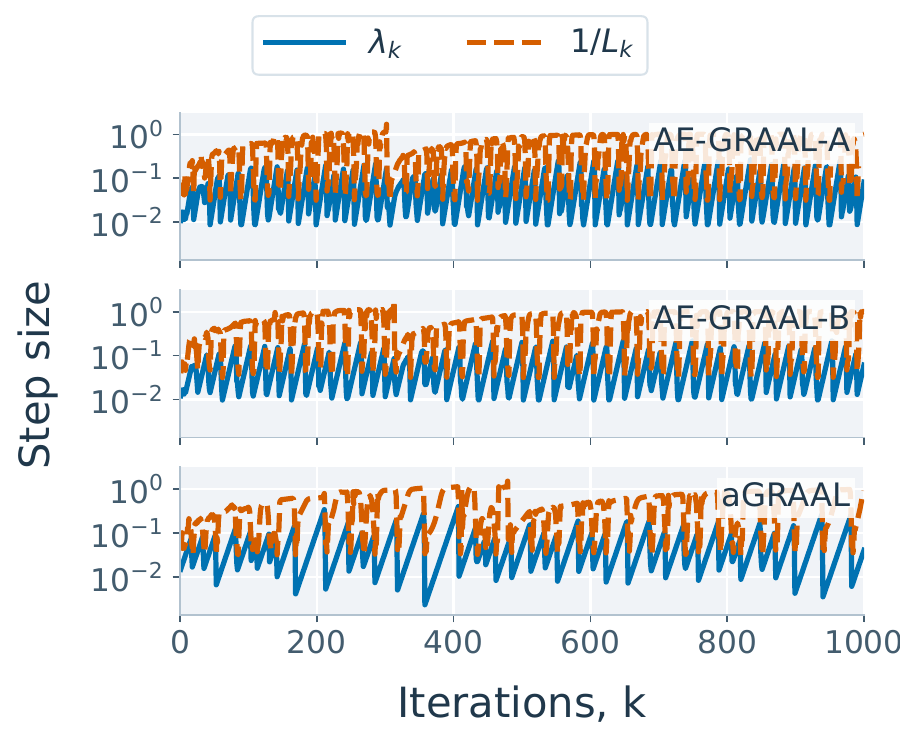}
    \hspace{1.1cm}
     \includegraphics[width=0.32\linewidth]{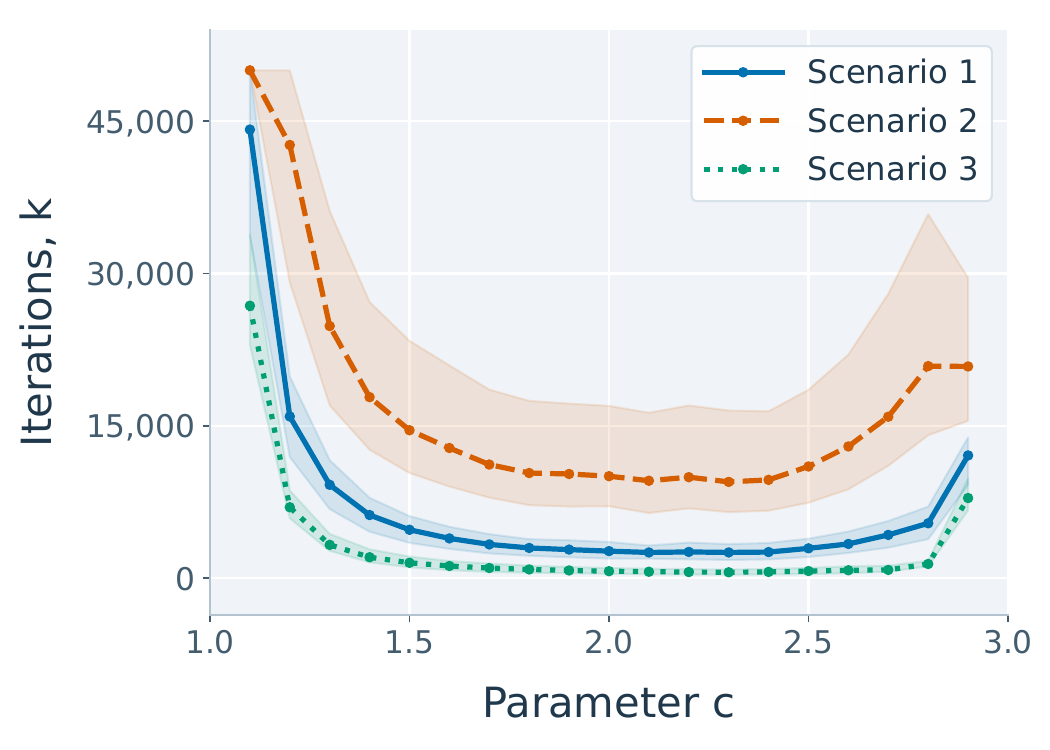}
\caption{Results for the Nash--Cournot equilibrium problem:
comparisons of $\lambda_k$ and $1/L_k$ for the first realisation
(seed $0$) of Scenario~1 (top left), Scenario~2 (top right),
and Scenario~3 (bottom left),
and median iteration counts for AE-GRAAL over
$c\in\{1.1,1.2,\ldots,2.9\}$ in all three scenarios (bottom right). In the bottom right case, each run stops when $R_k\leq10^{-9}$ or $50000$ iterations are reached.}
  \label{fig:nash_equilibrium_steps_all_run}
\end{figure}
\noindent
Throughout this experiment, we use $n=1000$, 
$\varepsilon_1=10^{-3}$, and coefficients are drawn independently from
$a_i\sim\mathcal U(1,100)$ and $\ell_i\sim\mathcal U(0.5,5)$. The first two scenarios follow the parameter choices of \cite{malitsky2020golden}:
\textbf{Scenario~1} combines $\gamma=1.1$ with
$\beta_i\sim\mathcal U(0.5,2)$, and
\textbf{Scenario~2} combines $\gamma=1.5$ with
$\beta_i\sim\mathcal U(0.3,4)$.
In \textbf{Scenario~3} we take $\gamma=1.5$ and
$\beta_i=1/2$ for every $i$. In this case, the marginal costs
are quadratic, since $(\ell_iq_i)^{1/\beta_i}=\ell_i^2q_i^2$.

\noindent
For each scenario, we generate $10$ realisations using seeds
$0,\ldots,9$ and start every algorithm from $x^1=(1,\ldots,1)$. The remaining initial parameters are
$x^0=P_{C_{\varepsilon_1}}(x^1+10^{-9}\xi)$, 
$\xi\sim\mathcal N(0,I_n)$, together with
$\lambda_0=
\|x^1-x^0\|/\|F(x^1)-F(x^0)\|.$
At each iteration, we record the natural residual
$R_k=\norm{x^k-
P_{C_{\varepsilon_1}}\bigl(x^k-F(x^k)\bigr)}.$
We stop each run when $R_k\leq10^{-9}$ or $50\,000$ iterations are reached. Figures~ \ref{fig:nash_experiment} and \ref{fig:nash_equilibrium_steps_all_run} collect the results for this experiment. From Figure~\ref{fig:nash_experiment}, we can observe that two instances of Algorithm~\ref{Alg:general} tend to require fewer iterations to reach the target accuracy than aGRAAL on these instances. Further, Figure~\ref{fig:nash_equilibrium_steps_all_run} suggests that this may be due to better approximation of the inverse local Lipschitz constant of $F$. These observations are specific to the tested scenarios and do not imply uniformly faster residual decay.

\subsection{Sparse logistic regression}
We next consider the sparse logistic regression problem
\[
\min_{x\in\mathbb R^n}J(x)
:=\sum_{i=1}^{m}\log\bigl(1+\exp(-b_i a_i^{\mathsf T}x)\bigr)
+\tau\|x\|_1,
\]
where $a_i\in\mathbb R^n$ are feature vectors and
$b_i\in\{-1,1\}$ are labels. Following \cite{malitsky2020golden},
we set $\tau=0.005\|A^{\mathsf T}b\|_\infty$, where $A$
has rows $a_i^{\mathsf T}$ and
$b=(b_1,\ldots,b_m)^{\mathsf T}$. This problem fits
\eqref{eq:monotone_vi} with $g(x)=\tau\|x\|_1$ and
\[
F(x)=-\sum_{i=1}^{m}
\frac{b_i a_i}{1+\exp(b_i a_i^{\mathsf T}x)}.
\]
Here, $F$ is monotone and globally Lipschitz with the constant
$L=\frac{\|A\|_2^2}{4}$, where $\|A\|_2$ is the spectral norm of $A$.
\begin{figure}[htbp]
    \centering
    \includegraphics[width=0.3\linewidth]{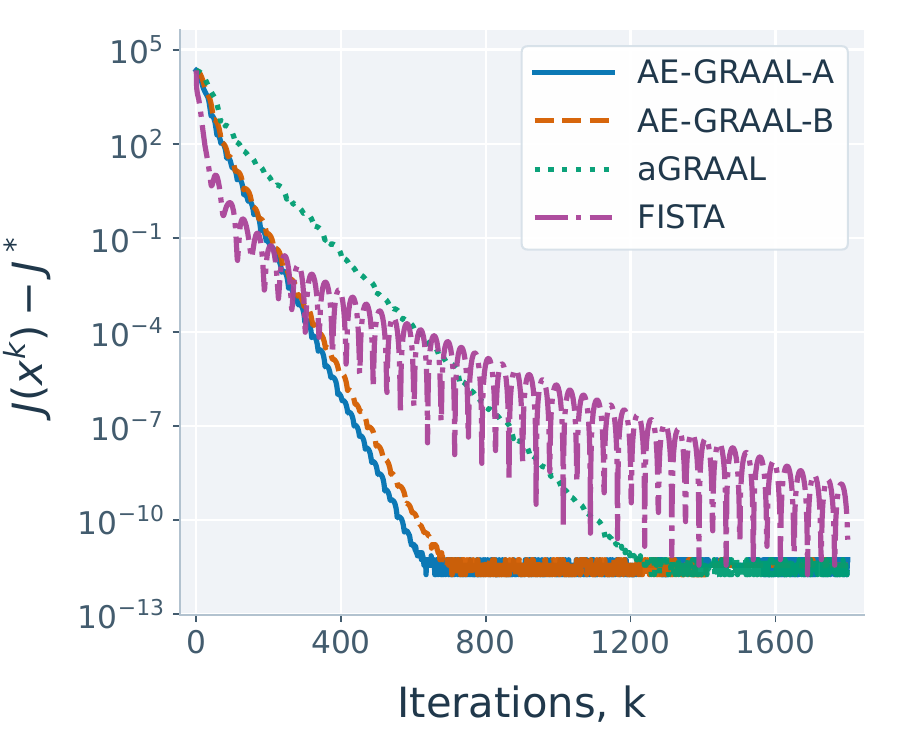}%
    \hfill
    \includegraphics[width=0.3\linewidth]{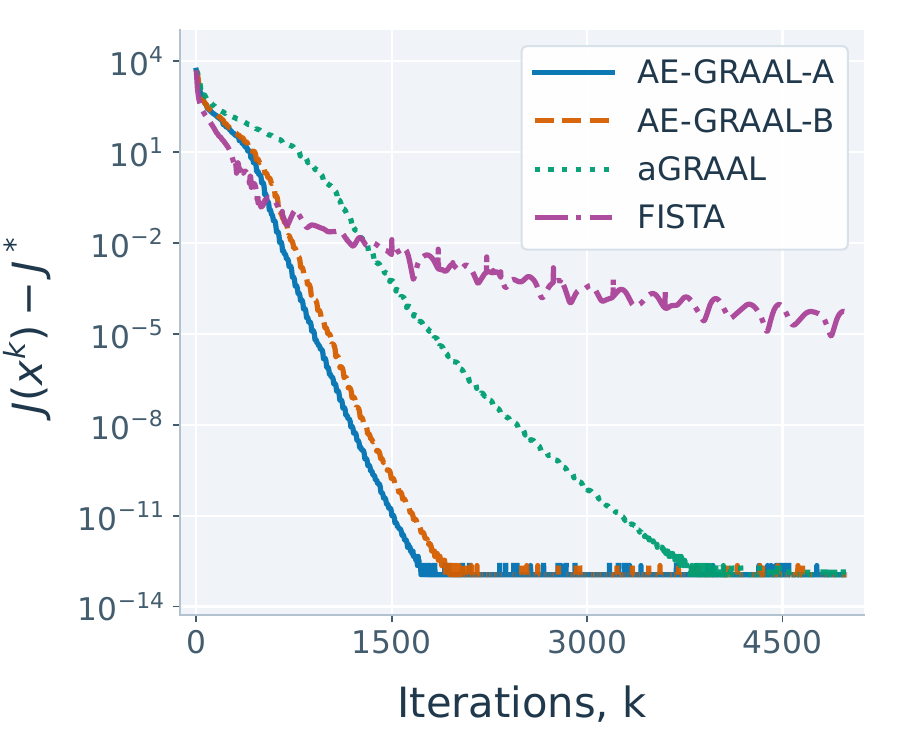}%
    \hfill
    \includegraphics[width=0.3\linewidth]{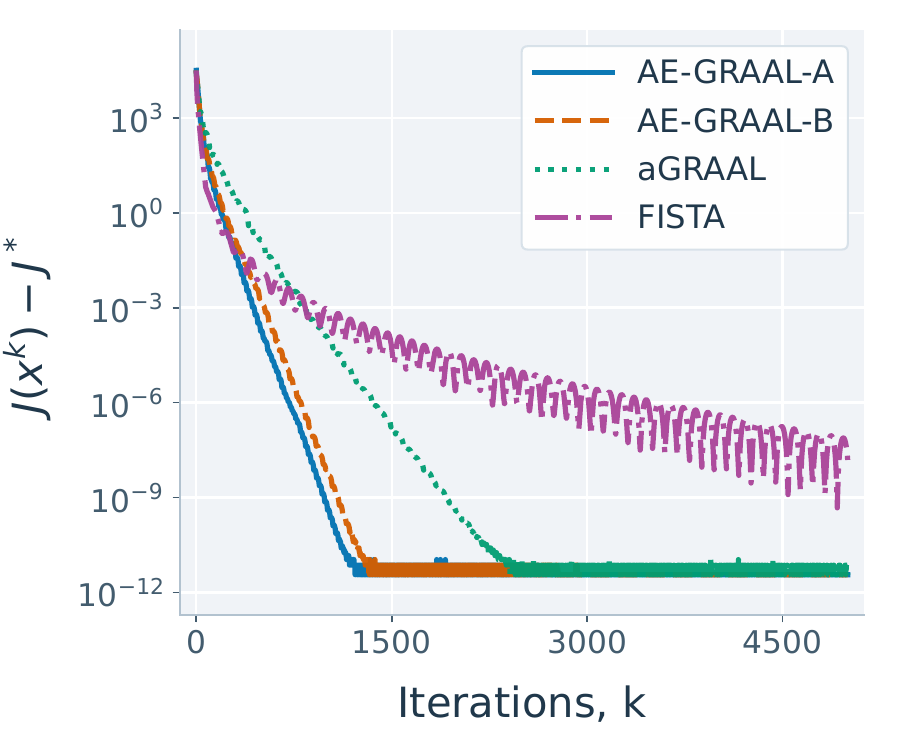}\\[1ex]
    \includegraphics[width=0.3\linewidth]{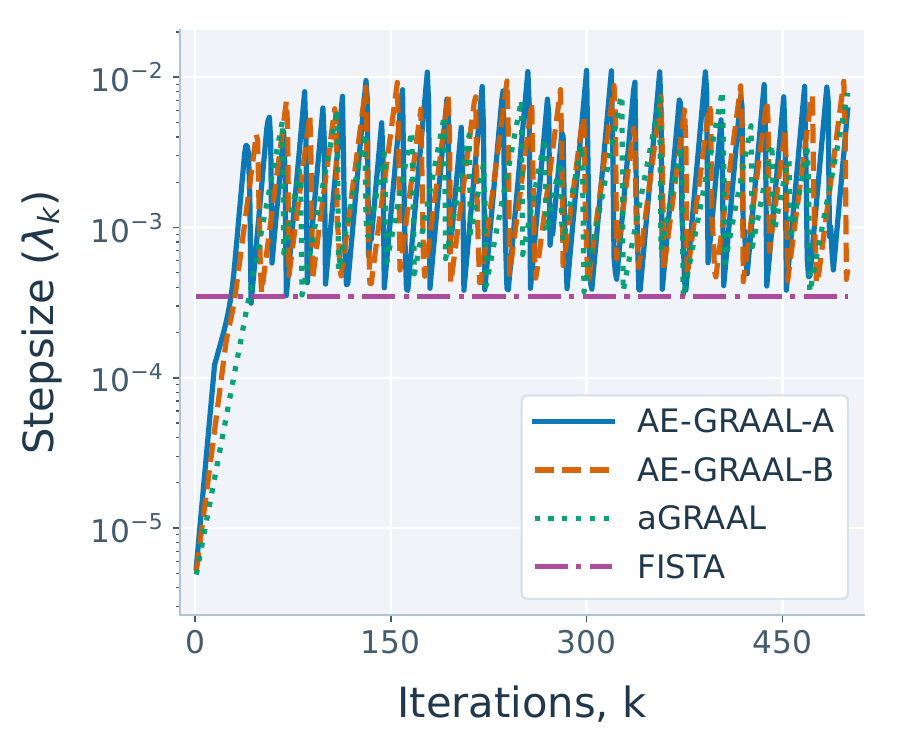}%
    \hfill
    \includegraphics[width=0.3\linewidth]{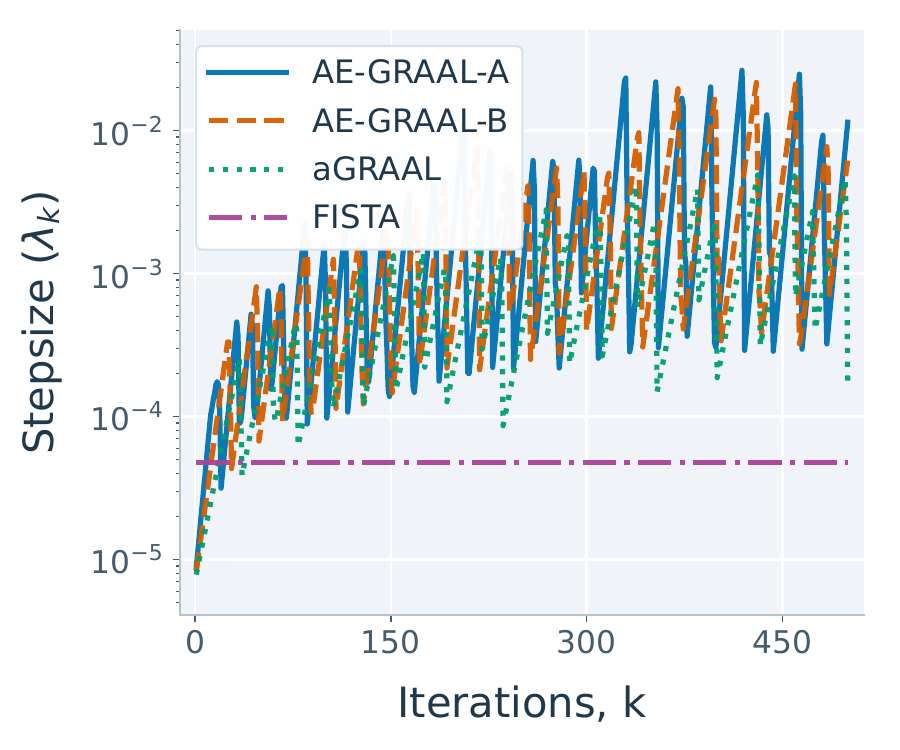}%
    \hfill
    \includegraphics[width=0.3\linewidth]{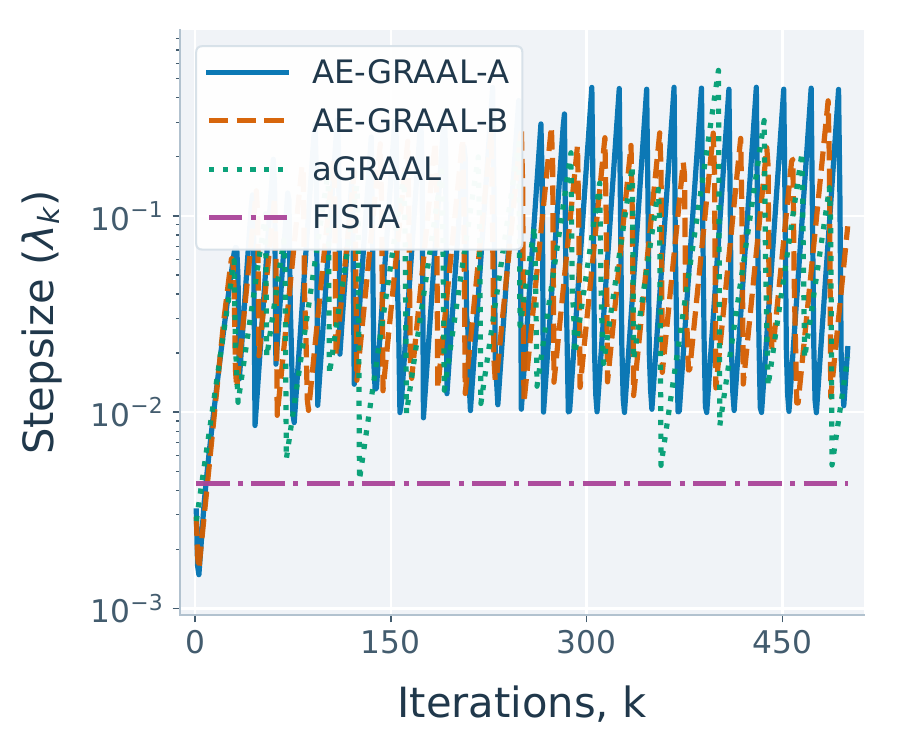}
    \caption{Results for sparse logistic regression
    on \texttt{ijcnn1}, \texttt{mushrooms}, and \texttt{real-sim} datasets,
    from left to right. Top row: objective gap $J(x^k)-J_{\mathrm{*}}$ versus iterations;
    Bottom row: the first $500$ step sizes.}
    \label{fig:logistic-datasets}
\end{figure}
\noindent
The proximal map is given componentwise by $[\operatorname{prox}_{\lambda g}(x)]_j
=\operatorname{sign}(x_j)\max\{|x_j|-\lambda\tau,0\}.$ We use the standard LIBSVM datasets \texttt{ijcnn1}, \texttt{mushrooms}, and
\texttt{real-sim}, with dimensions
$(m,n)=(49\,990,22)$, $(8\,124,112)$, and
$(72\,309,20\,958)$, respectively. We compare AE-GRAAL-A, AE-GRAAL-B, aGRAAL, and FISTA, all starting from $x^1=\mathbf0$. For the adaptive methods, we also use $x^0\sim\mathcal N(0,I_n)$ and
$\lambda_0=0.005\|x^1-x^0\|/\|F(x^1)-F(x^0)\|$. For FISTA, we take the fixed step $4/\|A\|_2^2$, where $\|A\|_2$ is the
spectral norm of $A$. Following~\cite{malitsky2020golden}, we plot $J(x^k)-J_*$, where $J_*$
is the smallest objective value recorded across the methods,
accepted only if at least one recorded point satisfies $\|x-\operatorname{prox}_g(x-F(x))\|\leq10^{-9}.$ We also record the first $500$ applied step sizes. From Figure \ref{fig:logistic-datasets}, we can observe that in all three datasets, both instances of Algorithm \ref{Alg:general} give a smaller objective gap faster than aGRAAL for the same number of iterations. Although step sizes continue to oscillate within a fixed range across all adaptive algorithms, our proposed methods adapt more effectively to the problem's geometry.

\subsection{Cubic complementarity problem}

We consider the problem of finding $x^* \in \R_+^n$ such that
\begin{equation}\label{eq:cubic reguralisation}
   \ip{F(x^*)}{x-x^*\geq0}
   \quad\forall x\in \R_+^n,
\end{equation}
where $F(x)=x+x^{\odot3}+Sx-b,$ with $(Sx)_i=x_{i+1}-x_{i-1}$ (indices taken cyclically) and $n\geq3$. Here
$x^{\odot3}$ denotes componentwise cubing, and $S^\top=-S$. 
\begin{figure}[htbp]
    \centering
    \includegraphics[width=0.3\linewidth]{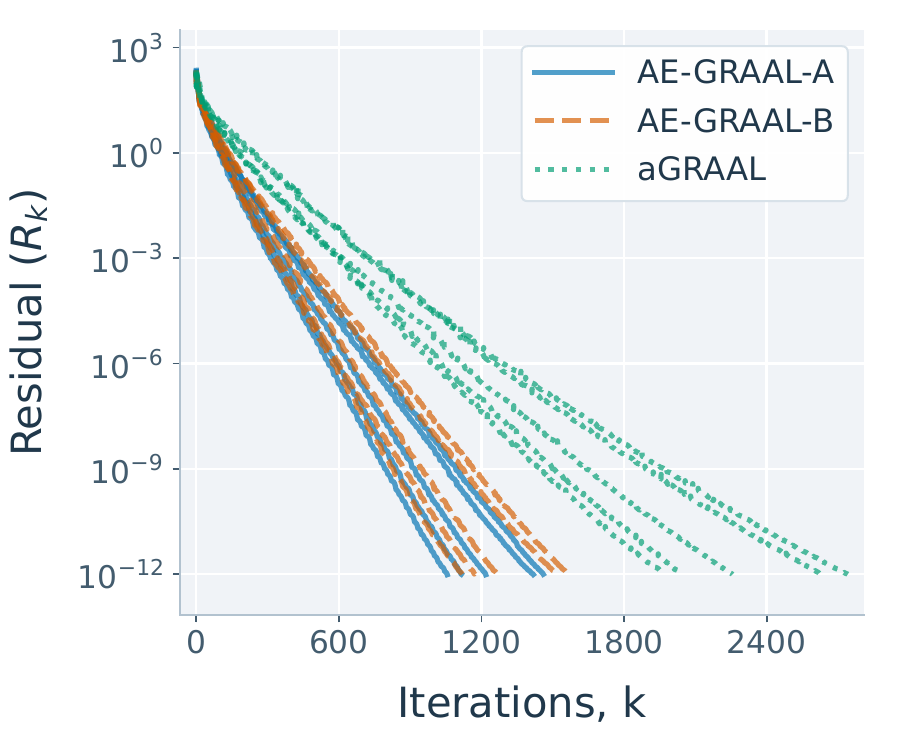}%
    \hfill
    \includegraphics[width=0.3\linewidth]{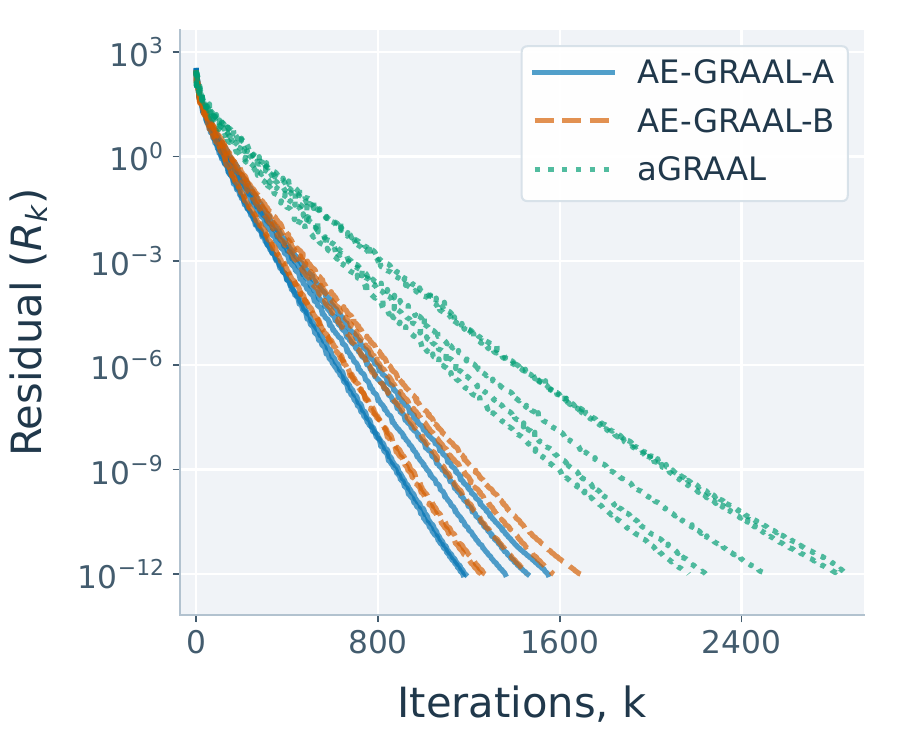}%
    \hfill
    \includegraphics[width=0.3\linewidth]{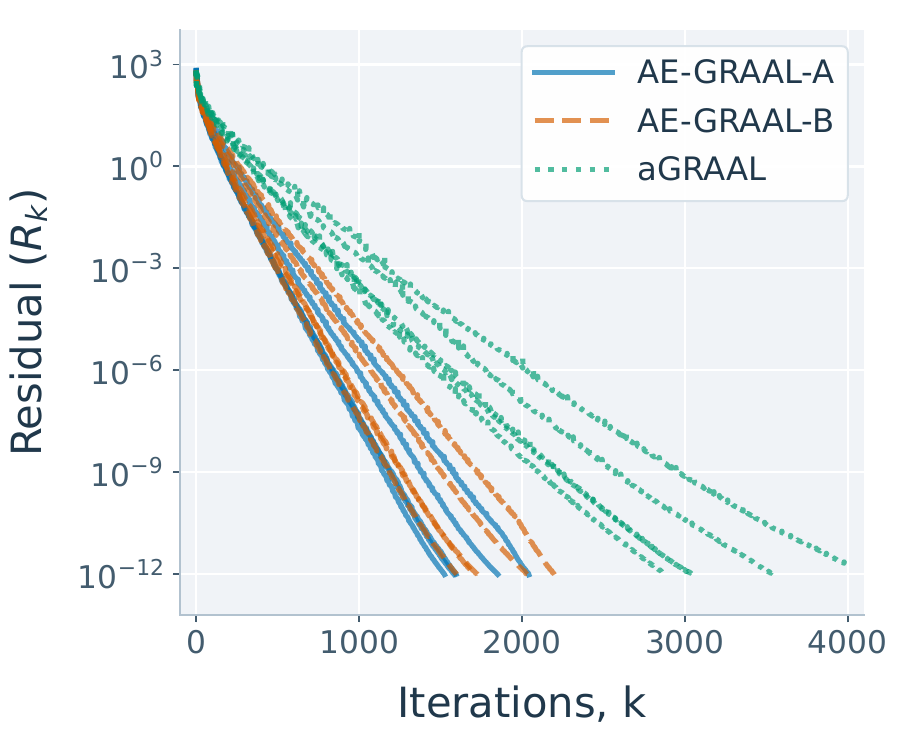}
    \caption{Residual ($R_k$) over iterations for the cubic complementarity problem with $n=5000$ (left), $n=10000$ (middle),
and $n=50000$ (right).}
    \label{fig:cubic problems}
\end{figure}
\begin{figure}[htbp]
  \centering
  \includegraphics[width=0.32\linewidth]{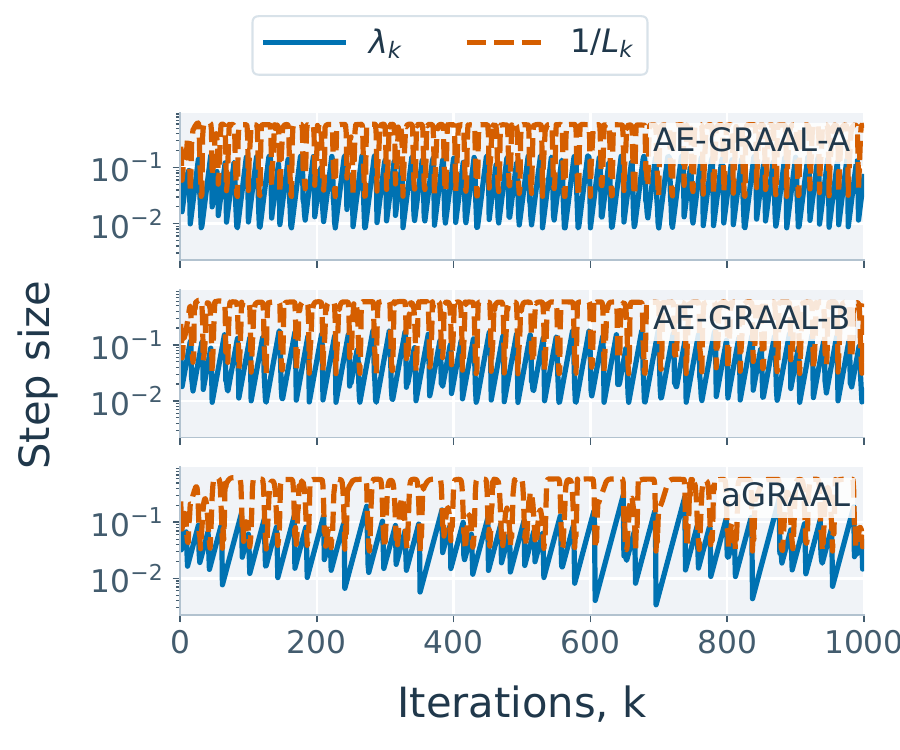}%
  \hspace{1.1cm}
  \includegraphics[width=0.32\linewidth]{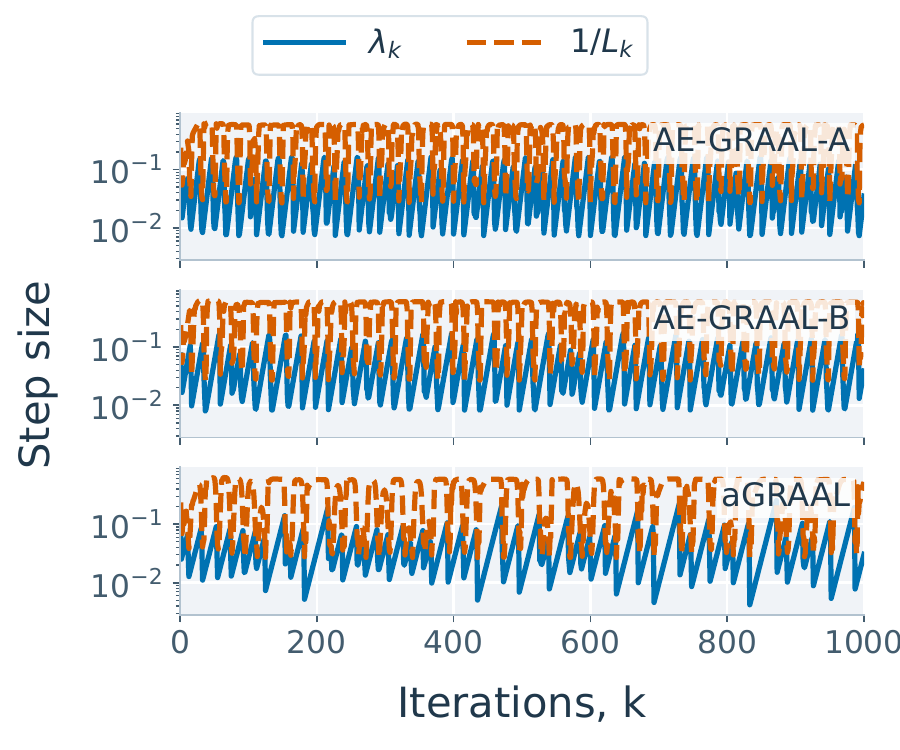}\\
    \includegraphics[width=0.32\linewidth]{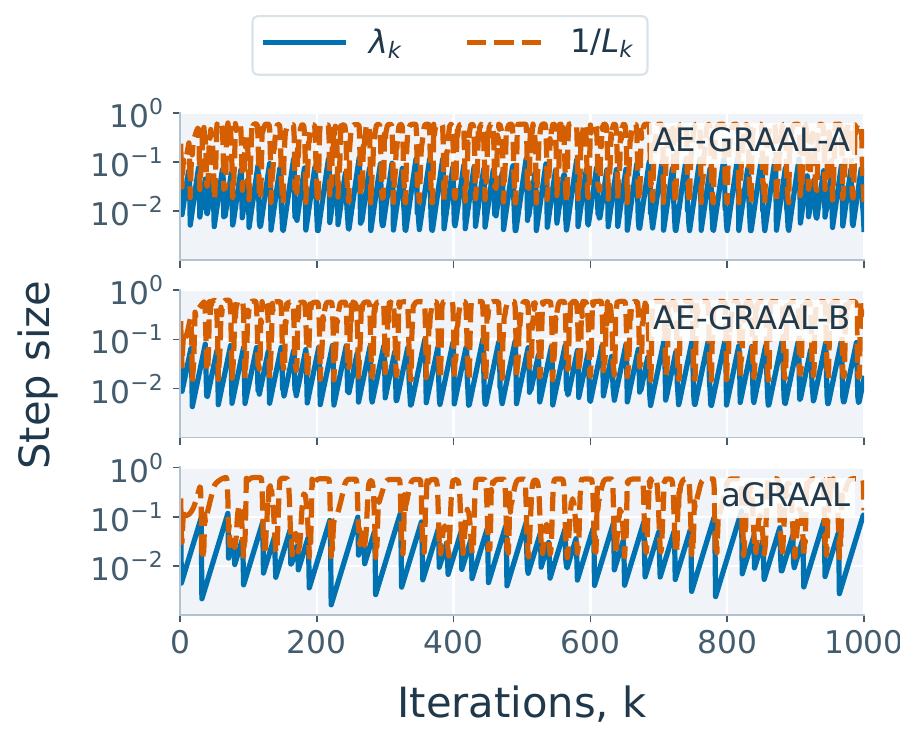}
    \hspace{1.1cm}
     \includegraphics[width=0.32\linewidth]{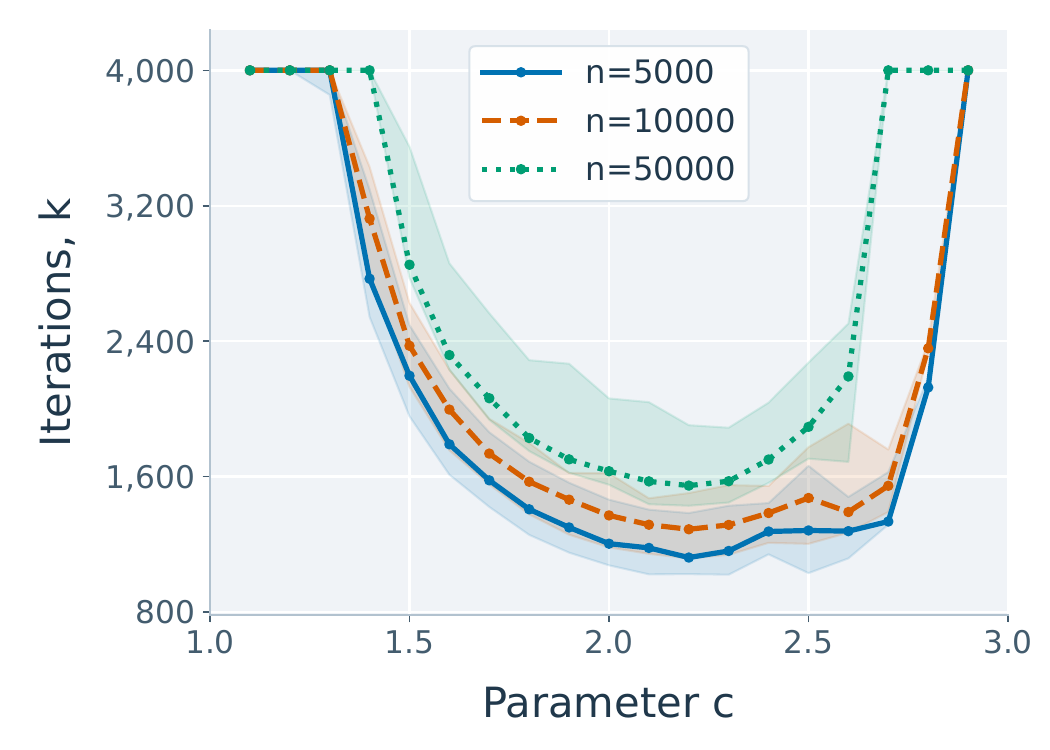}
\caption{Results for the cubic complementarity problem:
comparisons of $\lambda_k$ and $1/L_k$ for the first realisation
(seed $0$) with $n=5000$ (top left), $n=10000$ (top right),
and $n=50000$ (bottom left), and median iteration counts for
AE-GRAAL over $c\in\{1.1,1.2,\ldots,2.9\}$ at these three
dimensions (bottom right). In the bottom right case, each run stops when $R_k\leq10^{-12}$ or $4000$ iterations are reached.}
  \label{fig:cubic_complementarity_steps_all_run}
\end{figure}
The map $F$ is locally Lipschitz, but it is not the gradient of a scalar potential because the Jacobian of $F$ has a nonzero skew-symmetric part. To generate an instance with a known solution, we first draw $v\sim N(0,I_n)$
and set $x^*=[v]_+$, $s=[-v]_+$, and $ b=x^*+(x^*)^{\odot3}+Sx^*-s,$ where $[z]_+:=\bigl(\max\{z_i,0\}\bigr)_{i=1}^n$ is the componentwise positive
part. This choice gives $x^*\geq0$, $s\geq0$, and $x_i^* s_i=0$ for
every $i$, and it is chosen precisely so that $F(x^*)=s$. Hence $x^*$ is a solution of \eqref{eq:cubic reguralisation}. Throughout this experiment, we take $n\in\{5000,10000,50000\}$ and, for each dimension,
run each method five times with seeds $0,\ldots,4$.

\medskip
All algorithms are initialised at $x^1=(1,1,\ldots,1)$ and $x^0=[x^1+10^{-3}\xi]_+$, where
$\xi\sim N(0,I_n)$, and use the initial step size
$\lambda_0=\|x^1-x^0\|_2/\|F(x^1)-F(x^0)\|_2$. As in the previous experiments, we compute the natural residual $R_k$ and run each method until $R_k\le10^{-12}$
holds or a total of $5000$ iterations is reached. All algorithms in Figure~\ref{fig:cubic problems} reached the error threshold before exhausting the maximum number of iterations. Nevertheless, both AE-GRAAL-A and AE-GRAAL-B show faster residual decay than aGRAAL. Furthermore, Figure~\ref{fig:cubic_complementarity_steps_all_run} shows that our step size rule approximates the local inverse Lipschitz constant of $F$ more accurately than aGRAAL. Moreover, one can note from Figure~\ref{fig:cubic_complementarity_steps_all_run} that across all three scenarios, AE-GRAAL requires fewer iterations to reach the prescribed residual when $c$ lies somewhere between $2$ and $2.3$.

\section{Conclusions and future directions}
The results of this paper suggest several directions for further study.
\begin{enumerate}
    \item \textbf{Primal--dual extension.}
    It is natural to investigate whether the present Lyapunov estimates can be adapted to the primal--dual methods discussed in~\cite{chang2025splitting,chang2025primal,chang2026convex,soe2026golden,chang2022grpda,chang2021golden}. Furthermore, the golden ratio-based ADMM algorithms~\cite{chen2023golden,yan2025improved,yin2024golden} may also be revisited in light of the ideas discussed in this paper.

    \item \textbf{Bregman extension.}
    The Bregman extension of aGRAAL in \cite{tam2023bregman} suggests
    replacing the Euclidean proximal step with the one adapted to the geometry of the problem. However, their analysis requires that $F$ is locally Lipschitz~\cite[Theorem 4.1]{tam2023bregman} on the bounded set $\dom g\bigcap\operatorname{int}\dom h$, where $h$ is an appropriate Bregman kernel, to discuss the convergence of the iterates. It would be therefore interesting to
    see whether our adaptive analysis can be adapted to relax this restriction. In particular, it would be interesting to investigate whether a one-step Lyapunov analysis can replace the \cite[Eq. (37)]{tam2023bregman} with an alternating analysis discussed in this paper.

\item \textbf{Inertial methods and preconditioning.}
Adding an inertial term can accelerate the convergence of an optimisation algorithm, as demonstrated in~\cite{polyak1964some,boct2016inertial,moudafi2003convergence,chang2025splitting}. In the context of golden ratio-based algorithms, an interesting preconditioning technique was introduced in~\cite{wang2025new} that removes the dependence on the operator norm. For both algorithms proposed here (Algorithms~\ref{Alg:fixed} and~\ref{Alg:general}), it is therefore natural to ask whether such modifications can be incorporated to obtain methods that are efficient in practice.

    \item \textbf{Revisiting other GRAAL extensions.}
The further developments of GRAAL and aGRAAL published to date could be reconsidered in light of the idea proposed in this paper.

    \item \textbf{Stochastic and distributed methods.}
    A stochastic extension would be useful when exact operator
    evaluations are expensive or only noisy observations are
    available. In a distributed setting, each agent could use
    local observations and communication with its neighbours. New algorithms for distributed optimisation can be developed using the idea proposed in this work when at least one function involved in the objective is locally smooth.
\end{enumerate}

\section*{Acknowledgements}

The research of Santanu Soe was supported by the Prime Minister’s Research Fellowship program (Project number SB23242132MAPMRF005015), Ministry of Education, Government of India, and the Melbourne Research Scholarship. Matthew K. Tam was supported in part by Australian Research Council grant DP230101749.

\section*{Declarations}

\paragraph{Conflicts of interest.}
The authors declare that they have no competing interests
relevant to this work.

\paragraph{Code availability.}
The Python code used for the numerical experiments in this paper
is available at \url{https://github.com/santanusoemath/GRAAL_revisited}.

\section*{Declaration of AI use}

The authors used OpenAI's GPT-5.6 Sol to explore proof strategies, test counterexamples, and assist with generating an initial version of the code for the numerical experiments. The authors verified the mathematical arguments and code and take full responsibility for the content of this manuscript.

\bibliographystyle{abbrv}
\bibliography{GRAAL_revisited_references}

\appendix

\section{Parameter selection for the E-GRAAL algorithm}
\label{app:E-GRAAL_parameter_calculation}

Given \(c>1\), we seek \(\eta\geq0\) for which the
numerator of \(\Theta\) in \eqref{eq:descent_coefficient} and the Young weight $\tau:=\beta-2\lambda L$ are
positive. This implies that  \(\beta=1+\frac1c-\eta>0\). For convenience, we rewrite the numerator of \(\Theta\) in
\eqref{eq:descent_coefficient} as
\begin{equation}
N-2\lambda LD,
\label{eq:fixed_residual_numerator}
\end{equation}
where
\begin{align}
N&:=\beta\left(\eta+\frac2c\right)-1
=\frac{2c+2-c^2}{c^2}+\frac{c-1}{c}\eta-\eta^2,
\label{eq:fixed_N0}\\
D&:=\eta+\frac\beta{c^2}
=\frac{c^2-1}{c^2}\eta+\frac{c+1}{c^3}>0.
\label{eq:fixed_D}
\end{align}
Let $H_c(\eta):=\frac{N}{2D}.$ To obtain a descent inequality, we must have $N>0$; otherwise, \eqref{eq:fixed_residual_numerator} would be negative, and $\Theta$ could not be positive when the Young weight is positive. When \(N>0\), the restriction \(0<\lambda L<H_c(\eta)\) makes the numerator in
\eqref{eq:fixed_residual_numerator} positive. Furthermore, it also guarantees
the positivity of the Young weight. Indeed,
\begin{equation*}
\beta D-N
=\left(1-\frac\beta c\right)^2\geq0.
\end{equation*}
Therefore, \(2\lambda L<\frac{N}{D}\leq\beta\), and hence \(\tau:=\beta-2\lambda L>0\). It remains to maximise \(H_c(\eta)\) over \(\eta\geq0\)
and determine for which \(c>1\) this maximum is positive. Using the quotient rule, we have
\[
H_c'(\eta)=
\frac{\bigl(\frac{c-1}{c}-2\eta\bigr)D
-\frac{c^2-1}{c^2}N}{2D^2}.
\]
Expanding the numerator and then factoring it yields
\begin{align*}
&-\frac{c^2-1}{c^2}\eta^2-\frac{2(c+1)}{c^3}\eta
+\frac{(c^2-1)(c^2-2c-1)}{c^4}\notag\\
&=-\frac{c^2-1}{c^2}
\left(\eta+\frac{c-1}{c}\right)
\left(\eta-\frac{c^2-2c-1}{c(c-1)}\right).
\end{align*}
Consequently,
\begin{equation*}
H_c'(\eta)=
-\frac{c^2-1}{2c^2D^2}
\left(\eta+\frac{c-1}{c}\right)
\left(\eta-\frac{c^2-2c-1}{c(c-1)}\right).
\end{equation*}
For \(\eta\geq0\), all factors except the last have a fixed sign.
If \(1<c\leq1+\sqrt2\), the last factor is nonnegative. Thus
\(H_c\) is nonincreasing and its maximum is attained at \(\eta=0\).
If \(c>1+\sqrt2\), the maximiser is 
$\bar\eta:=\frac{c^2-2c-1}{c(c-1)}.$ This proves the selection of the parameters in \eqref{eq:fixed_eta_choice}. For the first range, substitution of $\eta=0$ in~$H_c(\eta)$ gives
\begin{equation}
H_c(0)=
\frac{(2c+2-c^2)/c^2}{2(c+1)/c^3}
=\frac{c(2c+2-c^2)}{2(c+1)}.
\label{eq:first_branch}
\end{equation}
It is positive on \(1<c\leq1+\sqrt2\), as the positive root
of \(2c+2-c^2\) is \(1+\sqrt3\).
For the second range, putting \(\eta=\bar\eta\) yields
\[
\beta=\frac2{c-1},\qquad
\eta+\frac2c=\frac{c^2-3}{c(c-1)}.
\]
Therefore, from \eqref{eq:fixed_N0} and \eqref{eq:fixed_D}, we obtain
\begin{align*}
N(c,\bar\eta)&=\frac{(3-c)(c-2)(c+1)}{c(c-1)^2}
\quad\text{and}\quad D(c,\bar\eta)=\frac{(c-2)(c+1)}{c^2},
\end{align*}
and
\begin{equation}
H_c(\bar\eta)=\frac{c(3-c)}{2(c-1)^2}.
\label{eq:second_branch}
\end{equation}
Observe that the numerator is positive exactly when \(c<3\). The two expressions agree
at \(c=1+\sqrt2\), where \(\bar\eta=0\).
Together, \eqref{eq:first_branch} and \eqref{eq:second_branch} yield~\eqref{eq:fixed_final_step_size}, and
\(0<\lambda L<\mu(c)\) ensures both
\(\beta>2\lambda L\) and \(\Theta>0\), as required in
Theorem~\ref{thm:fixed}.

\section{Parameter selection for the AE-GRAAL algorithm}
\label{app:local-parameter-calculation}
We justify the parameter selections
\eqref{eq:select_eta_from_the_range}--\eqref{eq:select_sigma_from_the_bound}.
To obtain strict descent from Lemma~\ref{lem:local_energy},
we require \(\beta>2\sigma\) and a positive lower bound for
\(\Delta(\theta)\) on \((0,\rho]\).
Recall from \eqref{eq:local_descent_term_general} and the definitions
following it that \(\beta=1+1/c-\eta\),
\(q(\theta)=\beta\eta+2\beta\theta/c-\theta^2\), and
\(D=\eta+\beta/c^2\). If \(\eta=0\), then
\(\beta=1+1/c\), and the numerator of
\(\Delta(\theta)\) tends to
\(-2\sigma\beta/c^2<0\) as $\theta$ decreases to zero. Thus,
\(\eta>0\) is necessary for this uniform descent estimate, and
\(D=\eta(1-1/c^2)+(c+1)/c^3>0\). Since \(q\) is a concave
quadratic, we have
\begin{equation*}
\inf_{0<\theta\leq\rho}q(\theta)=\min\{q(0),q(\rho)\}.
\end{equation*}
Thus, to get a positive lower bound for $\Delta$, it suffices to choose
\begin{equation}
\min\{q(0),q(\rho)\}>2\sigma D.
\label{eq:endpoint_conditions}
\end{equation}
Now we have to show that \eqref{eq:endpoint_conditions} implies $\tau:=\beta-2\sigma>0$. Note that $q(0)=\beta\eta>2\sigma D>0$. Hence $\beta>0$ together with
\(D=\eta+\beta/c^2>\eta\) gives
\(\beta>\frac{2\sigma D}{\eta}>2\sigma\).
Thus
\[
\Delta(\theta)\ge\delta_*=\frac{\min\{q(0),q(\rho)\}-2\sigma D}{\beta-2\sigma}
>0\quad\text{for}~0<\theta\le \rho,
\]
as used in \eqref{eq:uniform_delta}. It remains to determine for which \(c>1\) and \(\eta>0\) there exist
\(\rho>1\) and \(\sigma>0\) satisfying~\eqref{eq:endpoint_conditions}. By concavity of \(q\), the inequalities \(q(0)>0\),
\(q(\rho)>0\), and \(\rho>1\) imply \(q(1)>0\).
Conversely, if \(q(1)=\beta(\eta+2/c)-1>0\), then \(\beta>0\), and hence
\(q(0)=\beta\eta>0\). Therefore, continuity gives \(q(\rho)>0\)
for some \(\rho>1\), after which a sufficiently small
\(\sigma>0\) satisfies \eqref{eq:endpoint_conditions}. By completing the square, we obtain
\begin{equation*}
q(1)=\frac{3(c+1)(3-c)}{4c^2}
-\left(\eta-\frac{c-1}{2c}\right)^2.
\end{equation*}
For \(1<c<3\), this is positive exactly when
\(\eta_-(c)<\eta<\eta_+(c)\), where
\[
\eta_\pm(c):=
\frac{c-1\pm\sqrt{3(c+1)(3-c)}}{2c}.
\]
Together with \(\eta>0\), this gives
\eqref{eq:select_eta_from_the_range}.  To determine the growth factor $\rho$, let
\begin{equation*}
q(\rho)=-\left(\rho-\frac\beta c\right)^2
+\frac{\beta^2}{c^2}+\beta\eta.
\end{equation*}
The roots of $q$ are $r_\pm=\frac\beta c
\pm\sqrt{\frac{\beta^2}{c^2}+\beta\eta}.$ Since \(\beta\eta>0\), we have \(r_-<0<r_{+}\).
Also, \(q(1)>0\) implies \(r_{-}<1<r_{+}\). The factorisation \(q(\rho)=(r_+-\rho)(\rho-r_-)\) shows that the admissible values \(\rho>1\) form exactly the interval \((1,r_+)\) in \eqref{eq:select_rho_from_the_range}. Finally, choose
\[
0<\sigma<\sigma_{\max}=\frac{\min\{q(0),q(\rho)\}}{2D}.
\]
This is \eqref{eq:select_sigma_from_the_bound} which completes the construction.

\end{document}